\PassOptionsToPackage{table}{xcolor}
\documentclass[11pt]{article}

\RequirePackage{amsmath}
\RequirePackage{amsfonts}
\RequirePackage{amssymb}
\RequirePackage{hyperref}
\RequirePackage{multirow}
\RequirePackage{algorithm}
\RequirePackage{algpseudocode}
\RequirePackage{graphicx}
\RequirePackage{xcolor}
\RequirePackage{amsthm}
\RequirePackage{amsmath}
\RequirePackage{amsfonts}
\RequirePackage{amssymb}
\RequirePackage{hyperref}
\RequirePackage{multirow}
\RequirePackage{algorithm}
\RequirePackage{algpseudocode}
\RequirePackage{graphicx}
\allowdisplaybreaks

\usepackage{natbib}
\usepackage{dpr_fec,ifthen,dsfont}
\usepackage{notations}
\usepackage{array,tabularx,booktabs} 
\usepackage{graphicx,wrapfig,float,caption,subcaption} 
\usepackage{xcolor}
\theoremstyle{plain}
\newtheorem{theorem}{Theorem}[section]
\newtheorem{lemma}[theorem]{Lemma}
\newtheorem{corollary}[theorem]{Corollary}

\theoremstyle{definition}
\newtheorem{definition}[theorem]{Definition}

\newtheorem{assumption}[theorem]{Assumption}

\theoremstyle{remark}
\newtheorem{remark}{Remark}

\usepackage{isetechreport_yd}

\def\revisionno{2}
\def\originaldate{February 11, 2023}
\coraltrue
\cvcrfalse

\begin{document}

\title{\Large{A Variance-Reduced and Stabilized Proximal Stochastic Gradient Method with Support Identification Guarantees for Structured Optimization}}

\author{Yutong Dai\thanks{E-mail: yud319@lehigh.edu}}
\affil{Department of Industrial and Systems Engineering, Lehigh University}
\author{Guanyi Wang\thanks{E-mail: guanyi.w@nus.edu.sg}}
\affil{Department of Industrial Systems Engineering and Management, National University of Singapore}
\author{Frank E.~Curtis\thanks{E-mail: frank.e.curtis@lehigh.edu}}
\affil{Department of Industrial and Systems Engineering, Lehigh University}
\author{Daniel P.~Robinson\thanks{E-mail: daniel.p.robinson@lehigh.edu}}
\affil{Department of Industrial and Systems Engineering, Lehigh University}
\titlepage

\maketitle

\begin{abstract}
  This paper introduces a new proximal stochastic gradient method with variance reduction and stabilization for minimizing the sum of a convex stochastic function and a group sparsity-inducing regularization function. Since the method may be viewed as a stabilized version of the recently proposed algorithm \pstorm{}, we call our algorithm \spstorm{}. Our analysis shows that  \spstorm{} has strong convergence results. In particular, we prove an upper bound on the number of iterations required by \spstorm{}  before its  iterates correctly identify (with high probability) an optimal support (i.e., the zero and nonzero structure of an optimal solution).  Most algorithms in the literature with such a support identification property use variance reduction techniques that require either periodically evaluating an \emph{exact} gradient or storing a history of stochastic gradients. Unlike these methods, \spstorm{} achieves variance reduction without requiring either of these, which is advantageous. Moreover, our support-identification result for \spstorm{} shows that, with high probability, an optimal support will be identified correctly in \emph{all} iterations with index above a threshold. We believe that this type of result is new to the literature since the few existing other results prove that the optimal support is identified with high probability at each iteration with a sufficiently large index (meaning that the optimal support might be identified in some iterations, but not in others).  Numerical experiments on regularized logistic loss problems show that  \spstorm{} outperforms existing methods in various metrics that measure how efficiently and robustly iterates of an algorithm identify an optimal support.
\end{abstract}

\section{INTRODUCTION}
We consider the regularized stochastic learning problem
\begin{align}\label{prob:main}
\min_{x\in\R{n}} F(x) := & ~ f(x) + r(x),
\end{align}
where $f(x):=\Embb_{\xi\sim \Pcal}[\ell(x;\xi)]$ with $\xi$ being a random vector following a probability distribution $\Pcal$, $\ell(\cdot,\xi)$ is a smooth convex function almost surely with respect to the distribution of $\xi$, and $r$ is a sparsity-promoting closed convex function with group separable structure, i.e., $r(x) := \sum_{i=1}^{\ngrp}r_i([x]_{g_i})$ for some number of groups $\ngrp > 0$ with $g_i \subseteq\{1,2,\dots, n\}$ for each $i\in\{1,2,\dots,\ngrp\}$,  $\bigcup_{i=1}^{\ngrp} g_i=n$, and $g_i\cap g_j=\emptyset$ for all $i\neq j$. Some commonly used regularization functions have these properties, such as the  weighted $\ell_1$ norm $\sum_{i=1}^n\lambda_i| [x]_i|$ and the weighted non-overlapping \grplone{} norm $\sum_{i=1}^{\ngrp}\lambda_i\norm{[x]_{g_i}}$, where $\{\lambda_i\}$ are positive scalars, $[x]_i$ denotes the $i$th component of $x$, $[x]_{g_i}$ denotes the subvector of $x$ with entries from $g_i$, and $\|\cdot\|$ is the $\ell_2$ norm. Problem~\eqref{prob:main} is general enough to cover a broad class of problems of interest. In particular, when data samples ($\xi$) are available in a streaming manner, problem~\eqref{algo:main} recovers online convex learning~\citep{hazan2016introduction}, and when $\Pcal$ is a uniform distribution over a finite set $\{1,2,\cdots,N\}$, problem~\eqref{prob:main} recovers many  regularized finite-sum problems~\citep{tibshirani1996regression, hastie2009elements}. 

In this work, we are interested in designing an algorithm for solving problem~\eqref{prob:main} that can identify the support of an optimal solution (i.e., the zero and nonzero group structure of an optimal solution) in a finite number of iterations. This can be useful for applications like variable selection in regression problems~\citep{tibshirani1996regression}. 
It can also be used in combination with higher-order methods to design more efficient algorithms. For example, subspace acceleration methods benefit from $F$ being smooth over the variables in the support of an optimal solution, even though $F$ may be non-differentiable over the entire set of variables.  For such methods, once the support is identified, more powerful methods (e.g., truncated Newton's method~\citep{nocedal2006numerical}) can be applied over the variables in the support to accelerate the local convergence rate~\citep{wright2012accelerated,chen2017reduced, curtis2022}.



\subsection{Related Work}

The proximal stochastic gradient method~\citep{rosasco2020convergence} and its variants~\citep{xiao2014proximal,defazio2014saga,wang2019spiderboost,pham2020proxsarah, tran2022hybrid} are perhaps the most popular methods for solving problem~\eqref{prob:main}. Since there is a large body of work on proximal stochastic gradient methods, we will (in alignment with the contributions of our work) focus on methods that have both a convergence guarantee and support identification property. Support identification is also sometimes referred to as manifold identification~\citep{wright2012accelerated, poon2018local, sun19a, lee2012manifold}.

Proximal stochastic gradient-type methods are based on iterations that take the form
\begin{equation}\label{eq:generic.psg}
    y_{k+1} \gets \prox{\alpha_k r}{x_k-\alpha_k d_k} \text{ with } \alpha_k>0,
\end{equation}
where $\prox{\alpha_k r}{\cdot}$ is the proximal operator~\citep[Definition 6.1]{beck2017first} associated with $r$ and step size $\alpha_k > 0$ and $d_k$ is an estimator of $\grad f(x_k)$. If $d_k = \grad \ell(x_k;\xi_k)$ for some realization $\xi_k$ of the random variable $\xi$ and $x_{k+1}=y_{k+1}$, then \eqref{eq:generic.psg} recovers the proximal stochastic gradient method. 

As observed by \citet{poon2018local} and \citet{sun19a}, the proximal stochastic gradient method does not have a support identification property because the error in the stochastic gradient estimator $\epsilon_k=d_k-\grad f(x_k)$ does not vanish as $k$ goes to infinity. One way of overcoming this deficiency is to employ variance reduction techniques.  When $r$ is the weighted $\ell_1$-norm, \citet{sun19a} considers the variance reduction properties of \proxsvrg{}, \saga{}, and \rda{} (i.e, they consider whether $\E{}{\norm{\epsilon_k}} \to 0$),\footnote{These results for \proxsvrg{}, \saga{}, and \rda{} can be found in Table 2, Appendix C.3, and Appendix C.4 of \citep{sun19a}.} and establishes an active-set identification property (in expectation) for these three methods. Specifically, for a given  sufficiently large $k$, they show that the zero groups of $x_k$ agree with the zero groups of the optimal solution (in expectation).  Moreover, when $F$ is strongly convex so that a unique minimizer $x^*$ exists, by knowing the rates at which $\{\E{}{\norm{x_k-x^*}}\}$ and $\{\E{}{\norm{\epsilon_k}}\}$ converge to zero, \citet[Theorem 4]{sun19a} establishes an upper bound, that holds in expectation, on the number of iterations before the zero variables are identified. When $r$ is strongly convex, \citet{lee2012manifold} establishes for \rda{} that, for any given sufficiently large $k$, the support of $x_k$ matches that of $x^*$ with high probability.  (Observe that this means that the supports can match in some such iterations while not in other such iterations.) Later, \citet{huang2022training} extends this result for \rda{} to the non-convex setting by making additional assumptions on the rate of convergence of the iterates and the step sizes.  

A drawback of \proxsvrg{} and \saga{} is that they are only applicable when problem~\eqref{prob:main} has a finite-sum structure, i.e., $\Pcal$ is a uniform distribution over a finite set $\{1,2,\dots,N\}$.  In particular, \proxsvrg{} requires an extra \emph{exact} evaluation of $\grad f$ every epoch,
and \saga{} requires one \emph{exact} evaluation of $\grad f$ in the first iteration and stores a history of stochastic gradients in a matrix of size $N\times n$, where $N$ is the size of the data set and $n$ is the number of optimization variables.  Thus,  \proxsvrg{} and \saga{} are not practical for applications involving streaming data or large~$N$.

The recent work by~\citet{cutkosky2019momentum} and its extension by~\citet{xu2020momentum} consider a new stochastic gradient estimator called {\tt Storm}. When {\tt Storm} is combined with a proper step size selection strategy, it has a variance reduction property, and yet never requires an exact evaluation of $\nabla f$. Our method \spstorm{} draws inspiration from their work 
and introduces an iterate stabilization update to achieve a support identification property 
without having to store a history of stochastic gradients or to compute an exact evaluation of $\nabla f$. The above results are summarized in~\autoref{tab:existingworks}.

\begin{table}[!th]
\centering
\caption{The first column gives the algorithm name.  The second column shows the convergence rate of the iterates with $\rho_{\proxsvrg{}}>0$ and $\rho_{\saga{}}>0$. 
The third column shows the support identification complexity  
where $\Delta^*$ and $\delta^*$ are positive constants (see \eqref{def:Delta^*} and \eqref{def:delta^*}). (The $\Delta^*$ appearing in the result for our method \spstorm{} is a consequence of our accounting for both zero and nonzero groups, whereas the other results are derived based on when only the zero groups are identified.) 
The result for \rda{} is valid when $f$ and $r$ are both strongly convex whereas the result for \spstorm{} only assumes strong convexity of $f$. The fourth column indicates how often a method evaluates an exact gradient, and the fifth column gives the storage costs. The results for \proxsvrg{} and \saga{} hold only when problem~\eqref{prob:main} has a finite-sum structure.}
\begin{tabular}{c|llllll}
\hline
Algorithm   & $\norm{x_k-x^*}^2$                       & Support Identification                    & \# Exact $\nabla \!f$ & Storage     \\ \hline
\proxsvrg{} & $\Ocal\left(\rho_{\proxsvrg{}}^k\right)$ & $\Ocal(\log(1/\delta^*))$ & every epoch     & $\Ocal(n)$  \\
\saga{}     & $\Ocal\left(\rho_{\saga{}}^k\right)$     & $\Ocal(\log(1/\delta^*))$ & once             & $\Ocal(Nn)$ \\
\rda{}      & $\Ocal(\log k/ k)$                       &     $\Ocal\Big(\frac{1}{(\delta^*)^{4}}\Big)$       & never            & $\Ocal(n)$  \\
\rowcolor[HTML]{FFCE93} 
\spstorm{}     & $\Ocal(\log k/ k)$     & $\Ocal \Big(\max\Big\{\frac{1}{(\delta^*)^{4}},\frac{1}{(\Delta^*)^{4}}\Big\} \Big)$   & never            & $\Ocal(n)$  \\ \hline
\end{tabular}
\label{tab:existingworks}
\end{table}

\subsection{Contributions}
This paper makes three main contributions.
\begin{enumerate}
    \item We establish the variance reduction property (with high probability) of the {\tt{Storm}} stochastic gradient estimator (Theorem~\ref{thm:error.in.gradient}), which is missing in~\citet{xu2020momentum}. This is achieved by introducing a simple  stabilization step in line~\ref{line:stablelization} of  Algorithm~\ref{algo:main}, which we show allows for a constant step size to be employed.  This result is interesting in its own right, and the fact that our method allows for a constant step size to be used is a crucial property that we leverage in proving a support identification result.
    \item To the best of our knowledge, \rda{} and our proposed \spstorm{} are the only methods with a support identification property that neither require an exact gradient evaluation nor incur excessive storage costs. Compared with \rda{}, \spstorm{} has a stronger notion of support identification (formalized in Definition~\ref{def:id}). In particular, we show that, with high probability, \emph{all} sufficiently large iterates in \spstorm{} will correctly identify the support of the optimal solution. In contrast, \rda{} proves that each iterate with sufficiently large index will identify the support of the optimal solution with high probability (meaning that the support might be identified correctly in some iterations and not in others).
    We are able to obtain this stronger result as a consequence of the construction of the {\tt Storm} stochastic gradient estimator and the added stabilization step, which 
    allow for a sharp union bound (see Remark~\ref{remark:rda.explain} for additional details).
    \item Our numerical experiments on regularized logistic loss functions with weighted group $\ell_1$-norm regularization show that \spstorm{} outperforms popular methods in metrics that measure how efficiently and robustly iterates of an algorithm identify an optimal support, and in the final objective value achieved.
\end{enumerate}

\subsection{Notation and Preliminaries} 

Throughout the paper we use the following notation. We use $\norm{\cdot}$ to represent the $\ell_2$ norm, $|\Scal|$ to denote the cardinality of a set $\Scal$, and $\N{}_{+}$ and $\R{}_{+}$ to be the sets of positive integers and positive real numbers, respectively. For $N\in\N{}_{+}$, we define $[N] :=\{1,2,\cdots,N\}$.  For $x\in\R{n}$ and index set $\Ical\subseteq[n]$, we use $[x]_{\Ical}\in\R{|\Ical|}$ to denote the subvector of $x$ that corresponds to the elements of $\Ical$. 
For two sequences of non-negative real numbers $\{\phi_k\}_{k \geq 1}$ and $\{\psi_k\}_{k \geq 1}$, we say $\phi_k = \Ocal(\psi_k)$ if and only if there exist constants $k_0\in\N{}_{+}$ and $M\in\R{}_{+}$ such that $\phi_k\leq M\psi_k$ for all $k\geq k_0$.

Let us now formally define what we mean by the support,  a support identification property, and a consistent support identification property for a randomized algorithm.

\begin{definition}[support]
The support of a point $x\in\R{n}$ is denoted by $\Scal(x)$ and defined as
$$
\Scal(x) := \{i \in [\ngrp] ~|~ [x]_{g_i} \neq 0 \},
$$
where $\{g_i\}_{i=1}^{\ngrp}$ forms a non-overlapping partition of $[n]$. We say that $x\in\R{n}$ has optimal support if and only if $\Scal(x) = \Scal(x^*)$ for some solution $x^*\in\R{n}$ to~problem \eqref{prob:main}.  
\end{definition}\label{def:support}


\begin{definition}[support identification property]\label{def:one.iter.id}
A randomized algorithm is said to have the \emph{support identification property} if and only if there exists $K \in \N{}_{+}$ and $p \in (0,1]$ such that, when the algorithm generates a sequence of vectors $\{y_k\}_{k=1}^{\infty}$, one finds for each $k \geq K$ that the event $\{\Scal(y_k)=\Scal(x^*)\}$ occurs with probability at least $p$.
\end{definition}


\begin{definition}[consistent support identification property]\label{def:id}
A randomized algorithm has the \emph{consistent support identification property} if and only if there exist $K \in \N{}_{+}$ and $p \in (0,1]$ so that, when the algorithm generates a sequence of vectors $\{y_k\}_{k=1}^{\infty}$, one finds that the event $\mathcal{E}_{\text{id}} := \bigcap_{k\geq K}^{\infty}\{\Scal(y_k)=\Scal(x^*)\}$ occurs with probability at least $p$.
\end{definition}

While~\citet{lee2012manifold} and~\citet{sun19a} prove the support identification property of their algorithms (see  Definition~\ref{def:one.iter.id}), we prove the stronger consistent support identification property (see Definition~\ref{def:id}) for \spstorm{}.

We next introduce some properties related to the proximal operator. For any $\alpha>0$ and convex function $r$, the proximal operator $\prox{\alpha r}{\cdot}$ is single-valued. We define 
\begin{equation}\label{eq:chi.measure}
\chi(x;\alpha):=\tfrac{1}{\alpha}\norm{\prox{\alpha r}{x-\alpha\grad f(x)}-x},
\end{equation}
which is the norm of the so-called gradient mapping, and is known to serve as an optimality measure for problem~\eqref{prob:main}~\citep[Theorem 10.7 (b)]{beck2017first}.

\section{ALGORITHM}\label{sec.algorithm}

In this section, we present \spstorm{} as   Algorithm~\ref{algo:main} for solving problem~\eqref{prob:main}. At the beginning of iteration $k$, a mini-batch of independently and identically distributed (i.i.d) data samples $\{\xi_{k,i}\}_{i=1}^{m}$ are drawn according to the distribution $\Pcal$, and two stochastic gradients $v_k$ and $u_k$ are formed at the current iterate $x_k$ and the previous iterate $x_{k-1}$ in  \eqref{def.v}--\eqref{line:two.sg}. Then, the {\tt{Storm}} stochastic gradient estimator is constructed in  line~\ref{line:storm}. After performing the proximal stochastic gradient update to obtain $y_k$, a stabilization step is performed in line~\ref{line:stablelization}. As shown in the proof of Theorem~\ref{thm:error.in.gradient}, the stabilization step is critical because it allows for a constant step size strategy to be employed (i.e., $\alpha_k\equiv\ualpha > 0$ for all $k$), which in turn allows us to prove a consistent support identification result for \spstorm{}.

\begin{algorithm}[!th]
\caption{\spstorm{}} 
\label{algo:main}
\begin{algorithmic}[1]
    \State \textbf{Inputs:} Initial point $x_0=x_1\in\R{n}$, size of mini-batch $m \in \mathbb{N}_{+}$,  weight sequence $\{\beta_{k}\}_{k\geq2} \subset (0,1)$, stepsize sequence $\{\alpha_k\} \subset (0,\infty)$, and  parameter $\zeta \in (0,\infty)$.
    \For{$k = 1,2,\dots,$}
        \State Draw $m$ i.i.d samples $\{\xi_{k1}, \cdots, \xi_{km}\}$ w.r.t.~$\Pcal$.
          \State Set
          \vspace{-10pt}
         \begin{equation}\label{def.v}
            v_k \gets \tfrac{1}{m}\sum_{i=1}^m\grad \ell(x_{k};\xi_{ki}).
          \end{equation}
          \vspace{-15pt}
        \If{$k=1$}
          \State Set $d_k \gets v_k$.
        \Else
          \State Set
          \vspace{-15pt}
          \begin{equation}\label{line:two.sg}
            u_k \gets \tfrac{1}{m}\sum_{i=1}^m\grad \ell(x_{k-1};\xi_{ki}).
          \end{equation}
          \vspace{-15pt}
          \State Set $d_{k} \gets v_k + (1-\beta_k)(d_{k-1}-u_{k})$. \label{line:storm}
        \EndIf
        \State Compute $y_{k} \gets \prox{\alpha_{k} r}{x_{k}-\alpha_{k}d_{k}}$.\label{line:yk}
        \State Set $ x_{k+1} \gets x_k + \zeta\beta_k(y_{k}- x_{k})$. \label{line:stablelization}
    \EndFor
\end{algorithmic}
\end{algorithm}

\section{ANALYSIS} \label{sec.analysis}
 
 We begin this section by introducing the assumptions under which our convergence analysis is performed.

\subsection{Assumptions} 

Our first assumption concerns strong convexity of $f$ and Lipschitz continuity of the gradient of the loss function $\ell$. 
 
\begin{assumption}\label{ass:generic.assumptions}
The following hold: 
\begin{enumerate}
    \item $f$ is $\mu_f$-strongly convex over $\R{n}$ and $r_i$ is convex and closed over $\R{n}$ for all $i\in[\ngrp]$.
    \item There exists a constant $L_g>0$ such that, for any $(x,y)\in\R{n}\times\R{n}$ and any $\xi\sim \Pcal$, it holds that
    $$\norm{\grad \ell(x,\xi)-\grad \ell(y,\xi)}\leq L_g\norm{x-y},
    $$
    i.e., $\grad f$ is $L_g$-Lipschitz continuous.\label{ass:generic.assumptions.2}
\end{enumerate}
\end{assumption}
The strong convexity assumption on  $f$ is for deriving a complexity result for consistent support identification. This assumption can be relaxed to $f$ being convex if, similar to \citet{sun19a}, we instead assume that there exists a decreasing sequence $\{\nu_k\}$ such that $\Pmbb(\{\norm{x_k-x^*}\leq \nu_k\})=1$.  Under this assumption, we can also prove a consistent support identification result for \spstorm{}, although without an explicit upper bound on $K$ in Definition~\ref{def:id}---whereas under Assumption~\ref{ass:generic.assumptions} we provide such an upper bound. The smoothness assumption on $\ell(\cdot,\xi)$ is standard~\citep{cutkosky2019momentum,xu2020momentum}.

For our next assumption, we refer to the filtration---defined by the initial point and sequence of mini-batch stochastic gradients---corresponding to the stochastic process generated by the algorithm.  Denoting $\Fcal_1 := \sigma(x_1)$ and, for all $k \geq 2$, denoting $\Fcal_k$ as the $\sigma$-algebra generated by the random variables $\{\{\Xi_{1,i}\}_{i=1}^m,\dots,\{\Xi_{(k-1),i}\}_{i=1}^m\}$ (of which $\{\{\xi_{1,i}\}_{i=1}^{m},\dots,\{\xi_{(k-1),i}\}_{i=1}^{m}\}$ is a realization), it follows that $\{\Fcal_k\}$ is this filtration of interest. Recall that the distribution $\Pcal$ of $\xi$ is independent of the filtration.
\begin{assumption}\label{ass:proxstorm.varinace.reduced}
The following hold:
\begin{enumerate}
    \item For all $k \geq 1$, $\E{{\xi\sim\Pcal}}{\grad \ell(x_{k};\xi)~|~\Fcal_{k} }=\grad f(x_k)$.\label{ass:proxstorm.varinace.reduced.1}
    \item {There exists $G_r\in\R{}_{+}$ such that, for all $k \geq 1$, $\Pmbb_{} \{\|g_r\|_2 \leq G_r,\ \forall g_r \in \partial r(x_k)\} = 1$. \label{ass:proxstorm.varinace.reduced.3} }
    \item There exists $\sigma\in\R{}_{+}$ such that, for all $k \geq 1$, $\Pmbb_{\xi\sim\Pcal}\{{\norm{\grad \ell(x_k,\xi)-\grad f(x_k)}\leq \sigma}~|~\Fcal_k\}=1$.\label{ass:proxstorm.varinace.reduced.4}
    \item There exists $G_d\in\R{}_{+}$ such that, for all $k\geq 1$, $\Pmbb_{\xi\sim\Pcal}\{{\norm{d_k}\leq G_d}~|~\Fcal_k\}=1$.\label{ass:proxstorm.varinace.reduced.5}
\end{enumerate}
\end{assumption}

Assumption~\ref{ass:proxstorm.varinace.reduced}(\ref{ass:proxstorm.varinace.reduced.1}) ensures that the stochastic gradient $\grad \ell(x; \xi)$ is an unbiased estimator of the gradient $\grad f(x)$ for all $x\in\R{n}$. Assumption~\ref{ass:proxstorm.varinace.reduced}(\ref{ass:proxstorm.varinace.reduced.3}) provides a constant upper bound on the norm of an element of $\partial r(x)$ for all $x\in\R{n}$, which exists when $r$ is the weighted $\ell_1$-norm or weighted group $\ell_1$-norm, for example. Assumption~\ref{ass:proxstorm.varinace.reduced}(\ref{ass:proxstorm.varinace.reduced.4}) guarantees (almost surely) a bound on the difference between $\grad \ell(x_k; \xi)$ and $\grad f(x_k)$ for all $k\in \N{}_{+}$. This assumption is implied by the uniform bound assumption on $\grad \ell(x;\xi)$ used in~\citep[Assumption 4]{liu2022meta}. 
It may be possible to relax Assumption~\ref{ass:proxstorm.varinace.reduced}(\ref{ass:proxstorm.varinace.reduced.4}) by assuming that the stochastic gradient error has a sub-exponential tail, e.g., \citet{na2022hessian}, which we leave as future work. Assumption~\ref{ass:proxstorm.varinace.reduced}(\ref{ass:proxstorm.varinace.reduced.5}) is  implied by the following two, perhaps more natural, assumptions: (i) There exists a constant $c_e>0$ such that, for all $k$, it holds that $\Pmbb\{\norm{d_k-\grad f(x_k)}\leq c_e | \Fcal_k\} = 1$, i.e., the error in the stochastic gradient estimator $d_k$ is almost surely bounded; and (ii) There exists a constant $c_{\alpha}$ such that, for a given $\alpha>0$ and all $k\geq 1$, it holds that  $\Pmbb\{\chi(x_k;\alpha)\leq c_{\alpha} | \Fcal_k\}=1$ (also see~\eqref{eq:chi.measure}), i.e., the optimality measure is almost surely bounded. Note that Assumption~\ref{ass:proxstorm.varinace.reduced}(\ref{ass:proxstorm.varinace.reduced.5}) is slightly weaker than a bounded iterates assumption, which is also made in \rda{}~\citep{lee2012manifold}. A proof that Assumption~\ref{ass:proxstorm.varinace.reduced}(\ref{ass:proxstorm.varinace.reduced.5}) follows from (i) and (ii) can be found in Appendix~\ref{lemma:ass.implication}.

Our last assumption is on the parameters of  Algorithm~\ref{algo:main}.
\begin{assumption}\label{ass:algo.choice}
The sequences $\{\beta_k\}$ and $\{\alpha_k\}$ in Algorithm~\ref{algo:main} are chosen, with $c>1$ and $\ualpha\in (0,\infty)$, to satisfy $\beta_k= \min\{ 1/2, c/(k+1) \}$ and $\alpha_{k} \equiv \ualpha$ for all $k \geq 1$.
\end{assumption}
The constant $1/2$ appearing in the definition of $\beta_k$ in Assumption~\ref{ass:algo.choice} can be replaced by any constant between zero and one; the choice of $1/2$ is to simplify expressions appearing throughout our analysis.

\subsection{Convergence Analysis}

The first result establishes the variance reduction property of the {\tt Storm} stochastic gradient estimator.
\begin{theorem} \label{thm:error.in.gradient}

Let Assumption~\ref{ass:generic.assumptions}--Assumption~\ref{ass:algo.choice} hold, let $\epsilon_k=d_k-\grad f(x_k)$ for all $k \in \N{}_{+}$, and define $\underline{k}=\lceil (2c) - 1\rceil$.  Then, for any $k \geq \underline{k}$ and any $\eta_k \in (0,1)$, the event $\Ecal_k := \{\norm{\epsilon_k} \leq U(k)\}$ holds with probability at least $1 - \eta_k$, where for some constant $C\in\R{}_{+}$ independent of $k$, one defines
\begin{align*}
    U(k) =  & ~ C\big( \sigma + L_g  (G_r+G_d) \zeta \ualpha \big)  \cdot \max\left\{ \left( \frac{\underline{k} + 1}{k + 2}\right)^{c}, ~ \frac{c}{\sqrt{k+2}} \right\} \sqrt{\log\frac{2}{\eta_k}}.
\end{align*} 

\end{theorem}
The proof of Theorem~\ref{thm:error.in.gradient} is presented in Appendix~\ref{app:error.in.gradient}. 

\begin{remark}
The upper bound $U(k)$ in Theorem~\ref{thm:error.in.gradient} is independent of the mini-batch size $m$. This is due to the 
bound in Assumption~\ref{ass:proxstorm.varinace.reduced}(\ref{ass:proxstorm.varinace.reduced.4}) that holds almost surely. 
\end{remark}
\begin{remark}\label{rem:ek-error}
By setting $\eta_k =\frac{\eta_0}{k^2}$ for all $k \in \N{}_{+}$ with constant $\eta_0 \in (0, 6 / \pi^2)$, one obtains  $U(k) = \Ocal(\max\{\sqrt{\log k}/k^c, ~ \sqrt{\log k / k}\} )$ so that
$\{\norm{\epsilon_k}\} \to 0$ with high probability. This is formalized in the next result.
\end{remark}
\begin{corollary}\label{cor:variance.all.iters}
Let Assumption~\ref{ass:generic.assumptions}--Assumption~\ref{ass:algo.choice} hold, $\eta_k =\frac{\eta_0}{k^2}$ for all $k\geq 1$ with $\eta_0 \in (0, 6 / \pi^2)$, and $\Ecal_k$ be defined as in Theorem~\ref{thm:error.in.gradient}.
Then, the event $\Ecal := \bigcap_{k\geq\underline{k}}^{\infty} \mathcal{E}_k$ happens with probability at least $1 - \frac{\eta_0 \pi^2}{6}$.
\end{corollary}
The proof of Corollary~\ref{cor:variance.all.iters} can be found in Appendix~\ref{app:proof.for.corllary}.

Next, we establish the rate of convergence of the iterate sequence $\{x_k\}$ with high probability (for small $\eta_0$). 

\begin{theorem}\label{thm:iterate.complexity}

Let $\ualpha = \mu_f / L_g^2$, $\zeta \in (0, 2)$, $\theta\geq 2$, $c = (2\theta L_g^2) / (\zeta \mu_f^2) > 2$,  and $\underline{k}=\lceil 2c-1\rceil$. Set $\eta_k = \eta_0 / k^2$ for all $k\geq 1$ with $\eta_0 \in (0, 6 / \pi^2)$. Then, under Assumption~\ref{ass:generic.assumptions}--Assumption~\ref{ass:algo.choice}, there exists a constant $C_3\in\R{}_{+}$ independent of $k$, such that the event $\Ecal^x_{k}:=\left\{\norm{x_{k}-x^*}^2 \leq {\bar c}_1 \frac{\norm{x_{\underline{k}}-x^*}^2}{k^\theta} + {\bar c}_2\cdot\frac{\log\frac{2k}{\eta_0}}{k}\right\}$ with ${\bar c}_1:=(\underline{k} + 2)^\theta$ and ${\bar c}_2:=C_3\zeta\left(\frac{\mu_f^2}{L_g^4} + \frac{2}{L_g^2}\left(1+\frac{\mu_f}{L_g}\right)^2\right)(\sigma + L_g(G_r + G_d)\zeta \ualpha)^2$ satisfies
$$
\P{\bigcap_{k\geq\underline{k}}^{\infty}\Ecal^x_{k}} \geq 1 - \eta_0 \pi^2 / 6 > 0.
$$
\end{theorem}
The proof of Theorem~\ref{thm:iterate.complexity} is presented in Appendix~\ref{app:iterate.complexity}.

\begin{remark}\label{rem:xk_rate}
Theorem~\ref{thm:iterate.complexity} provides a $\Ocal(\sqrt{\log k/k})$ convergence rate for $\norm{x_k - x^*}$ for all $k\geq\underline{k}$ with high probability. It is worth noting that the constant $\underline{k}$ depends on the square of the condition number $L_g/\mu_f$. We also note that the first term $\cbar_1 \| x_{\underline{k}}-x^*\|^2/k^\theta$ can be made to converge to zero arbitrarily fast by choosing $\theta$ as large as desired, although this results in larger $\underline{k}$.  It is the second term $\cbar_2 \log(\frac{2k}{\eta_0}) / k$ that dictates the overall convergence rate of the iterates.  This rate of convergence is obtained by using the rate at which the error in the {\tt Storm} stochastic gradient estimator converges to zero (see Remark~\ref{rem:ek-error}). 
\end{remark}

\begin{remark}
Theorem~\ref{thm:iterate.complexity}  establishes a sub-linear rate of convergence for the iterates with high probability for strongly convex loss functions. However, it remains unknown whether there exists a method that has a linear convergence rate for strongly convex functions and avoids huge storage and exact gradient evaluations.   
\end{remark}

\subsection{Support Identification} \label{sec.analysis.id.support}
In this section, we restrict our attention to $r$ being the weighted non-overlapping group $\ell_1$ regularizer, i.e., $r(x)
=\sum_{i=1}^{\ngrp}\lambda_i\norm{[x]_{g_i}}$ with $\ngrp > 0$, $\{g_i\}\subseteq[n]$ for each $i\in[\ngrp]$,  $\bigcup_{i=1}^{\ngrp} g_i=[n]$, $g_i\cap g_j=\emptyset$ for all $i\neq j$, and $\{\lambda_i\}_{i=1}^{\ngrp}$ strictly positive group weights.

Let us now introduce quantities that are crucial for establishing our support identification result. Specifically, let $x^*$ be the unique solution to problem \eqref{prob:main}. Define 
\begin{align}
    \Delta &:= 
    \begin{cases}
    \displaystyle{\min_{i\in\Scal(x^*)}} \norm{[x^*]_{g_i}}  & \text{if $\Scal(x^*)\neq\emptyset$,} \\
    1 & \text{if $\Scal(x^*) = \emptyset$,}\\
    \end{cases}
    \nonumber
    \\[0.3em]
    \Delta^* &:= \min\{1,\Delta\}, \label{def:Delta^*}
    \\[0.3em]
    \delta_{\min} &:= 
    \begin{cases}
    {\displaystyle\min_{ i\not \in \supp{x^*}}}\{\lambda_i - \norm{\grad_{g_i} f(x^*)}\} &\text{ if } \supp{x^*} \subsetneqq [\ngrp],\\
    1 &\text{ if }  \supp{x^*} = [\ngrp],
    \end{cases}\nonumber
    \\[0.3em]
    \delta^* &:=\min\{\delta_{\min}, 1\}. \label{def:delta^*}
\end{align}
Geometrically, $\Delta$ captures the minimum $\ell_2$-norm of the groups that are non-zero at $x^*$, taking into account the possibility that $\supp{x^*}$ is empty. The definition of $\delta_{\min}$ measures the minimum distance between $\lambda_i$ and the corresponding optimal dual variables (see~\eqref{prob:prox.dual}) for groups not in $\Scal(x^*)$. 
To see this, without loss of generality, suppose that $\supp{x^*} \subsetneqq [\ngrp]$. 
For any $\alpha>0$ define $z^* := x^*-\alpha\grad f(x^*)$ and then  consider the proximal problem
\begin{align}
\min_{x\in \R{n}} \ \phi_p(x;x^*,\alpha) 
:= \tfrac{1}{2\alpha}\|x-z^*\|^2 + r(x)\label{prob:prox.primal}
\end{align}
and its dual problem
\begin{equation}\label{prob:prox.dual}
\max_{\omega\in \R{n}}
\ \phi_d(\omega;x^*,\alpha):=-\left(\tfrac{\alpha}{2}\|\omega\|_2^2+ \omega^T z^* \right)
\ \ \text{s.t.} \ \  r_*(\omega)\leq 1
\end{equation}
where
$r_*(\omega) = \max_{i\in[\ngrp]} \frac{\|[\omega]_{g_i}\|}{\lambda_i}$ is the dual norm of the weighted group $\ell_1$ norm.  
It can be seen that $x^*$ is the optimal solution to the primal problem~\eqref{prob:prox.primal}~\citep[Theorem 10.7]{beck2017first}. Denoting $\omega^*$ as the optimal solution to the dual problem~\eqref{prob:prox.dual}, it follows that
\begin{equation}\label{eq:prox.dual.sol}
\!\![\omega^*]_{g_i} = -\min\!\Big\{\tfrac{1}{\alpha}, \tfrac{\lambda_i}{\norm{[z^*]_{g_i}}}\Big\}[z^*]_{g_i}
\ \ \text{for all $i\in[\ngrp]$.} 
\end{equation}
Then, by the Fenchel-Young inequality~\citep[Theorem 31.1]{rockafellar1970convex}, 
it follows that
\begin{equation}\label{eq:linking-eq}
x^* = \alpha \omega^* + z^*.
\end{equation}
Combining the definition of $z^*$ and \eqref{eq:linking-eq}, one establishes that $\omega^*=\grad f(x^*)$. Therefore, $\delta_{\min}$ measures the minimum distance from $[\omega^*]_{g_i}$ to the boundary of the ball centered at origin with distance $\lambda_i$ for all $i\not\in\supp{x^*}$.

The discussion above leads to a non-degeneracy assumption: For  groups of variables not in $\supp{x^*}$, their corresponding dual variables are strictly feasible, i.e., $\norm{[\omega^*]_{g_i}}<\lambda_i$ for all $i\not\in\supp{x^*}$. Let us formally state this non-degeneracy assumption using $\omega^* = \grad f(x^*)$ to make it consistent with the literature~\citep{lee2012manifold, poon2018local, sun19a, curtis2022}.
\begin{assumption}\label{ass:support.id}
The scalar $\delta^*$ in~\eqref{def:delta^*} satisfies $\delta^*> 0$. 
\end{assumption}



With the non-degeneracy assumption in hand, we may now give a sufficient condition for support identification.
\begin{theorem} \label{thm:id.sufficient.cond}
Let Assumption~\ref{ass:support.id} hold.
Given $\alpha>0$, $d\in\R{n}$, and the optimal solution $x^*$ to problem~\eqref{prob:main}, let us define $z=x-\alpha d$ and $y=\prox{\alpha r}{z}$. If
$$
\norm{\frac{[z - x^*]_{g_i}}{\alpha}+\grad_{g_i} f(x^*)}<\delta^* \text{ for all } i\not\in \supp{x^*},
$$
then $\supp{y}\subseteq \supp{x^*}$. Furthermore, if $\norm{y-x^*}< \Delta^*$, then $\supp{x^*}\subseteq \supp{y}$ so that, in fact, $\supp{y} = \supp{x^*}$.
\end{theorem}

The proof of Theorem~\ref{thm:id.sufficient.cond} is presented in Appendix~\ref{app:id.sufficient.cond}.

\begin{remark}
Theorem~\ref{thm:id.sufficient.cond} extends the result in \citep[Lemma 1]{sun19a} from the $\ell_1$ regularizer to the group $\ell_1$ regularizer considered here. Also, our result slightly strengthens theirs since they only discuss the result  $\Scal(y) \subseteq \Scal(x^*)$.
\end{remark}

Using the sufficient conditions for support identification from  Theorem~\ref{thm:id.sufficient.cond}, the result of consistent support identification (Definition~\ref{def:id}) can now be established.

\begin{theorem} \label{thm:support.identification}

Let Assumption~\ref{ass:generic.assumptions}--Assumption~\ref{ass:support.id} hold, $\zeta \in (0, 2)$, $\theta\geq 2$, $c = (2\theta L_g^2) / (\zeta \mu_f^2) > 2$,  and $\underline{k}=\lceil 2c-1\rceil$. Consider the sequence $\{y_k\}$ of Algorithm~\ref{algo:main} and define the event $\Ecal_{k}^{\text{id}}=\{\Scal(y_k)=\Scal(x^*)\}$ for all $k\geq 1$. Then, there exists constants $\{C_{41}, C_{42}\}\subseteq\R{2}_{+}$ that are independent of $k$,
$k_{\delta^*}=(C_{41}/\delta^*)^{4}$ and $k_{\Delta^*}=(C_{42}/\Delta^*)^{4}$ such that, with $K := \max\{k_{\delta^*}, k_{\Delta^*}, \underline{k}$\}, it follows that
$$
\P{\bigcap_{k\geq K}^{\infty} \Ecal_{k}^{\text{id}} }\geq 1-\frac{\eta_0\pi^2}{6} > 0.
$$
\end{theorem}
The proof of Theorem~\ref{thm:support.identification} is presented in Appendix~\ref{app:support.identification}.
\begin{remark}\label{remark:rda.explain}
Using Theorem~\ref{thm:support.identification} and results from \citet{xiao2009dual} and \citet{lee2012manifold}, we can also derive a high probability support identification complexity bound for \rda{} for any given iterate $x_k$, which is different from the result in \citet[Theorem 5]{sun19a}. To do so, we need extra assumptions on the function $r$ that do not hold for the weighted group $\ell_1$-norm, and boundedness of $\{\nabla \ell(x_k;\xi_k)\}$ generated by \rda{}\footnote{See  Lemma~\ref{lemma:rda.related}(\ref{lemma:rda.related.2}) for precise details of the assumptions.}. Specifically, we consider the update of \rda{} as
$x_{k+1}=\prox{\alpha_kr}{-\alpha_kd_k} \text{ with } \alpha_k = \frac{\sqrt{k}}{\underline{\alpha}}$, 
where $d_k=\frac{1}{k}\sum_{i=1}^{k}\grad \ell(x_i; \xi_i)$\footnote{See  Lemma~\ref{lemma:rda.related}(\ref{lemma:rda.related.1}) to see how this form of the update is equivalent to the \rda{} update presented in~\citet{xiao2009dual}.}.
It follows from Lemma~\ref{lemma:rda.related}(\ref{lemma:rda.related.3}) that 
\begin{equation*}
\P{\supp{x_{k+1}} 
=\supp{x^*}}\geq 1 - \eta_k^{\rda{}}
\end{equation*}
where
\begin{equation*}
\eta_k^{\rda{}}
=\max\left\{\Ocal\left(\frac{1}{\delta^*\cdot k^{1/4}}\right), \Ocal\left(\frac{1}{\Delta^*\cdot k^{1/4}}\right)\right\}.
\end{equation*}
Since $\sum_{k=1}^{\infty}\eta_{k}^{\rda{}}$ diverges, one cannot give a lower bound on  $\P{\cap_{k\geq K^{\rda{}}}^{\infty}\{\supp{x_k}=\supp{x^*}\}}$ for some sufficiently large $K^{\rda{}}$. Instead, for any $\eta_0\in (0,1)$, there exists a $\kbar =  \Ocal\Big(\max\Big\{\left(\frac{1}{\eta_0\delta^*}\right)^4,  \left(\frac{1}{\eta_0\Delta^*}\right)^4 \Big\}\Big)$ such that any given $k \geq \kbar$ satisfies $\P{\Scal(x_{k+1})=\Scal(x^*)}\geq 1-\eta_0$. This establishes the support identification property (see Definition~\ref{def:one.iter.id}). However, in Theorem~\ref{thm:support.identification} we show that \spstorm{} has a consistent support identification property (see Definition~\ref{def:id}), which is a stronger result. Lastly, we note that the $K$ value appearing in Theorem~\ref{thm:support.identification} grows with the condition number $L_g/\mu_f$.
\end{remark}
\begin{remark}
Similar to~\citet{sun19a}, under additional assumptions, it is possible to extend Theorem~\ref{thm:support.identification} to the case that $f$ is convex. In particular, if we assume that $\norm{x_k-x^*}\leq A_k$ for some optimal solution $x^*$ and a decreasing sequence $\{A_k\}$ with some positive probability (for example, with probability $1-\eta_k$) for all $k\geq 1$, then we can prove a support identification result, but we no longer have a complexity bound.
\end{remark}

\section{NUMERICAL EXPERIMENTS}\label{sec.experiment}

\subsection{Problems, Baselines, and Implementation Details}\label{sec.problem.and.datasets}
\textbf{Problems.}~We consider solving problem~\eqref{prob:main} with $f(x)$ and $r(x)$ given by the regularized binary logistic loss and group-$\ell_1$ regularizer, respectively, resulting in the problem
$$
\min_{x \in \mathbb{R}^{n}} \tfrac{1}{N} \sum_{j=1}^{N} \log \left(1+e^{-y_{j} x^{T} d_{j}}\right) + 10^{-5}\|x\|^2 + \sum_{i=1}^{\ngrp}\lambda_i\norm{[x]_{g_i}}
$$ 
where $N$ is the number of data points, $d_j\in\R{n}$ is the $j$th data point, $y_j\in \{-1, 1\}$ is the class label for the $j$th data point, and $\lambda_i>0$ for all $(j,i)\in [N]\times[\ngrp]$. 
Data sets for the logistic regression problems were obtained from the LIBSVM repository.\footnote{https://www.csie.ntu.edu.tw/cjlin/libsvmtools/datasets} We excluded all multi-class (greater than two) classification datasets, datasets with feature less than 50 or samples less than 10000, and all data sets that were too large ($\geq$ 16GB)\footnote{Memory usage is counted by a Python object instead of the raw txt file. We also exclude the dataset \href{https://www.csie.ntu.edu.tw/~cjlin/libsvmtools/datasets/binary/epsilon_normalized.xz}{epsilon} since we had an error message indicating a wrong data format in line 33334.}. Finally, for the adult data (a1a--a9a) and webpage data (w1a--w8a), we used only the largest instances, namely a9a and w8a. This left us with our final subset of $10$ data sets that can be found in Table~\ref{tab:test-db}. Following~\citet{xiao2014proximal}, we scaled each data point to have a unit norm, i.e., $\norm{d_j}=1$ for all $j\in[N]$.


\begin{table}[thbp]
\centering
\caption{Description of the data sets. }
\label{tab:test-db}
\begin{tabular}{|l|r|r|}
\hline
data set           & N          & n         \\ \hline
a9a                & 32561      & 123        \\
avazu-app.tr       & 12,642,186 & 1,000,000 \\
covtype     & 581,012    & 54           \\
kdd2010            & 8,407,752 	& 20,216,830 \\
news20      & 19,996     & 1,355,191  \\
phishing           & 11,055     & 68     \\ 
rcv1        & 20,242     & 47,236      \\
real-sim           & 72,309     & 20,958 \\  
url      & 2,396,130 	& 3,231,961     \\
w8a                & 49,749     & 300  \\\hline
\end{tabular}
\end{table}

For each dataset, we considered four group structures and two different solution sparsity levels, which led to 80 test instances in total.  We considered the four different numbers of groups in $\{\lfloor 0.25n\rfloor, \lfloor 0.50n\rfloor, \lfloor 0.75n\rfloor, n\}$, where $n$ is the problem dimension; notice that the last setting recovers $\ell_1$-norm regularization. Then, for a given number of groups, the variables were sequentially distributed (as evenly as possible) to the groups; e.g., $10$ variables among $3$ groups would have been distributed as $g_1 = \{1,2,3\}$, $g_2 = \{4,5,6\}$, and $g_3 = \{7,8,9,10\}$. 
We considered two different solution sparsity levels obtained by adjusting the group weights $\{\lambda_i\}$. 
Specifically, we considered group weights $\lambda_i= \Lambda\sqrt{|g_i|}$ for all $i\in[\ngrp]$ with  $\Lambda=0.1\Lambda_{\min}$ and $\Lambda=0.01\Lambda_{\min}$, where $\Lambda_{\min}$ is the minimum positive number such that the solution to the logistic problem with $\lambda_i = \Lambda_{\min}\sqrt{|g_i|}$ is $x=0$. See~\citet[equation (23)]{Yang2015} for the formula to compute $\Lambda_{\min}$. 

\noindent\textbf{Baselines.} We choose \proxsvrg{}~\citep{xiao2014proximal}, \saga{}~\citep{defazio2014saga}, and \rda{}~\citep{xiao2009dual} as baselines since they have theoretical guarantees for identifying the support. We also include \pstorm{}~\citep{xu2020momentum} to demonstrate the empirical importance of the modification we made in \spstorm{} (i.e., the stabilization step in Line~\ref{line:stablelization}).  
We use \farsagroup{}~\citep{curtis2022}, a deterministic second-order method, to find a highly accurate estimate to the optimal solution $x^*$ for each test instance by solving the problem to high accuracy ($10^{-8}$), as measured by the norm of the  gradient mapping in~\eqref{eq:chi.measure}.

\noindent\textbf{Implementation Details}
We implemented a version of \proxsvrg{} as described in \citet[Equation (8) Option II]{poon2018local}, \saga{} as described in \citet[Equation (6)]{poon2018local}, \rda{} as described in \citet[Algorithm 1]{lee2012manifold}, and \pstorm{} as described in \citet[Algorithm 1]{xu2020momentum}\footnote{The code is publicly available at \url{https://github.com/Yutong-Dai/S-PStorm}.}.
(i)~\textbf{Step size strategy}: For \proxsvrg{}, \saga{}, and \spstorm{}, we used a constant step size strategy by setting $\alpha_k\equiv 0.1/L_g$, which follows the choice made in \citet{xiao2014proximal}. We remark that $L_g$ can be estimated by $1/4$ since the data set is normalized instance-wise (see \citet[Section 4.1]{xiao2014proximal} for the reason). For \rda{}, the step size was set as $\alpha_k=\sqrt{k}/\gamma$.\footnote{The original paper used $\beta_k$ to denote the step size. See part (1) of Lemma~\ref{lemma:rda.related} for how to map $\beta_k$ to $\alpha_k$.} We tuned $\gamma$ by choosing its value from the set $\{10^{j}\}_{j\in\{-4,-3,\ldots,2\}}$ using the 32 test instances obtained from the datasets a9a, covtype, phishing, and w8a, and found that $\gamma=10^{-2}$ worked the best. For \pstorm{}, we used  $\alpha_k=\frac{4^{1/3}/(8L_g)}{(k+4)^{1/3}}$ as suggested in \citet[Theorem 2]{xu2020momentum}. (ii)~\textbf{Algorithm specific parameters}: \proxsvrg{} is a double loop algorithm and we set the inner loop length to 1, i.e., the parameter $P$ in \citet[Equation (8) Option II]{poon2018local} was set to 1. For \rda{} the prox-function $h$ was chosen as the square of the $\ell_2$ norm. For \pstorm{} we used $\beta_k = \frac{1+24\alpha_k^2L_g^2-\frac{\alpha_{k+1}}{\alpha_k}}{1+4\alpha_k^2L_g^2}$, and for \spstorm{} we used  $\beta_k = \frac{1}{k+1}$. The $\zeta$ parameter is chosen in an adaptive way to improve the practical performance. In particular, $\zeta$ is initialized to $1$ and increased by $1$ after an  iteration is completed. Although this choice is not covered by the convergence theory, one could cap the number of adjustments made to $\zeta$, in which case it is covered by the theory.
For all algorithms, the batch size was set to 256 and the starting point was the zero vector.  
(iii) \textbf{Termination conditions}: A test instance was terminated when either 1000 epochs was reached,
or a 12 hour time limit was reached. We note that \saga{} terminated immediately on all test instances associated with the datasets avazu-app.tr, kdd2010, news20, real-sim, and url because the storage of the gradient look-up table exceeded the memory limit. 

\subsection{Numerical Results}
\label{sec:num-results-metrics}

Experiments were run on a cluster with 16 AMD Opteron Processor 6128 2.0 GHz CPUs and 32 GB memory.

\noindent\textbf{Support Identification Performance.} We considered four metrics for measuring an algorithm's performance on support identification. Specifically, we computed the supports of the iterates $\{x_{kb}~|~k=1\cdots,1000\}$ with $b=\lceil N/m \rceil$, where $m$ was the mini-batch size. The sequence $\{x_{kb}\}$ can be thought of as the ``major iterates" resulting after each full data-pass.
The first metric was the \textit{total number of identifications}, which measured the number of iterates in  $\{x_{kb}\}$ that correctly identified the support $\Scal(x^*)$ (the larger the better); the second metric was the \textit{first identification}, which was the smallest $k_0\in[1000]$ such that $x_{k_0b}$ identified the support $\Scal(x^*)$ (the smaller the better); the third metric was the \textit{first consistent identification}, which was the smallest $K\in[1000]$ such that all $\{x_{kb}\}_{k\geq K}$ identified the support $\Scal(x^*)$ (the smaller the better); the last metric was the \textit{last iterate support recovery}, which was defined as  $1-\frac{\left|\Scal(x_{1000b})\Delta \Scal(x^*)\right|}{\left|\Scal(x^*)\right|}$ (the closer to $1$ the better) with $\Delta$ being the set symmetric difference. 
The \textit{last iterate support recovery} metric was introduced because we observed that all five algorithms failed to identify the support $\Scal(x^*)$ on some test instances generated by the larger datasets (e.g., url) as a result of not getting an accurate enough approximate solution.
Nonetheless, when the algorithms terminated, the last iterates still had sparse structure, 
and the \textit{last iterate support recovery} metric measured how close the algorithm was to identifying the true support.  

For every test instance solved by a given algorithm, we repeated the experiments for 3 independent runs and for each run compute the four metrics, which are then averaged to obtain the final values of the metrics for the algorithms. For a given test instance and metric, we assigned scores from $\{1,2,3,4,5\}$ to the 5 algorithms based on their ranked performances. The better an algorithm performed, the higher the score it received. The best performer received a score of $5$, the second best performer received a score of $4$, and so forth.\footnote{When two or more algorithms obtained the same value for a metric, we assign them all the same score.}   For the first three metrics, if an algorithm failed to identify the support before it terminated, we assigned the algorithm a score of $0$. For each metric, we summed over all test instances to get the final scores for each algorithm and then normalized the scores so that the scores for all algorithms under a given metric summed to one.

We present the normalized scores for the 5 algorithms over the 4 metrics in \autoref{fig:suppot.id.performance}, and provide the raw data for these metrics in Appendix~\ref{appendix:exp.additional}. One can see that \spstorm{} consistently outperformed the other algorithms on all 4 metrics by a significant margin.

\begin{figure}[!th]
    \centering
    \includegraphics[width=0.5\textwidth]{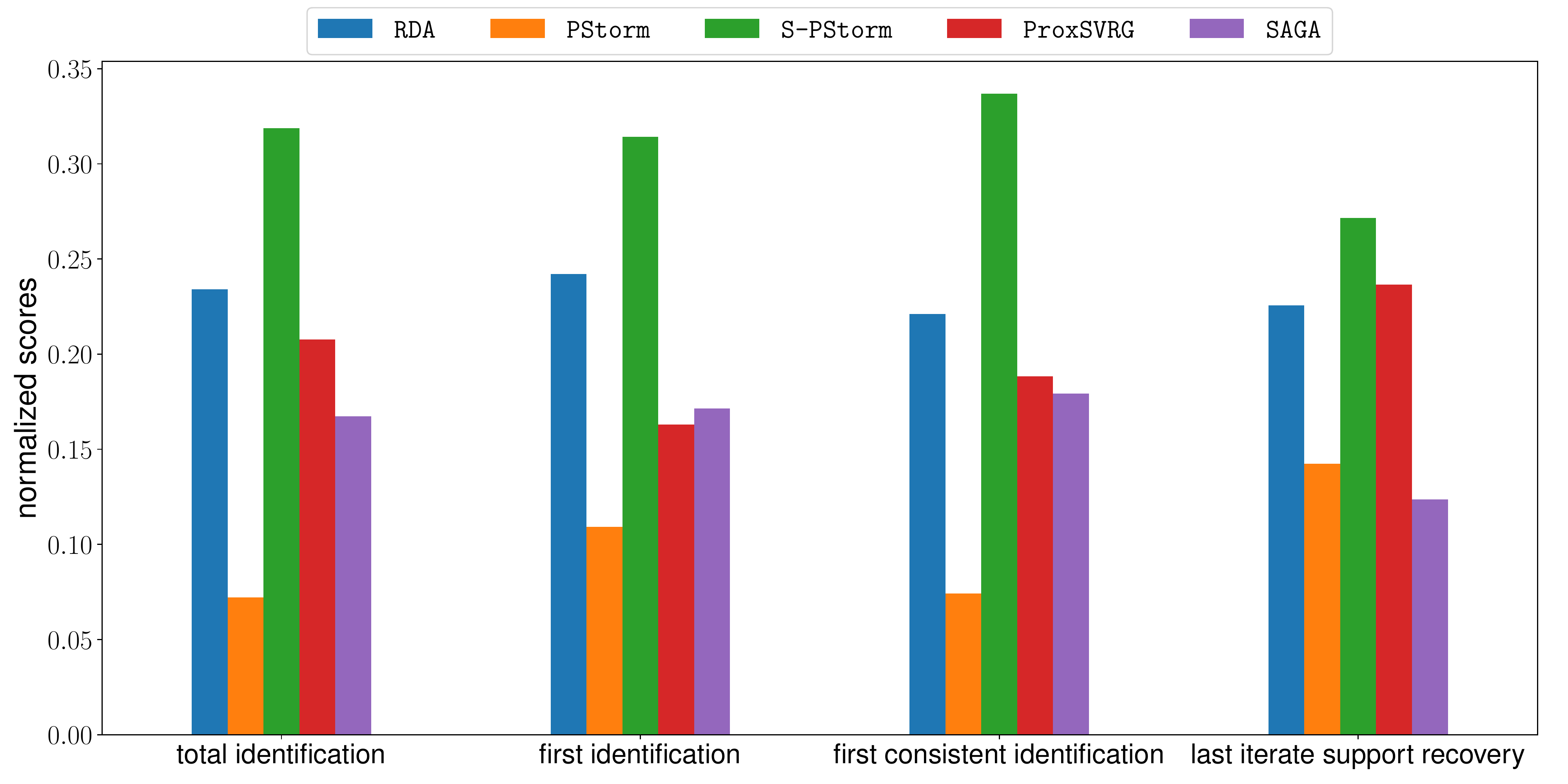}
    \caption{Normalized scores for four metrics that evaluate the performance of the support identification. }
    \label{fig:suppot.id.performance}
\end{figure}

\noindent\textbf{Solution Quality.} We measure the solution quality of an algorithm by computing the optimal objective function value gap. Specifically, for a given test instance, denote $F^*=\min_{j}\{F_{j}^{\text{best}}\}$, where $F_{j}^{\text{best}}=\min_{b\in[1000]}\{F(x_{kb}^{j})\}$ with $j\in$\{\proxsvrg{}, \saga{}, \rda{}, \pstorm{}, \spstorm{}\} and $\{x_{kb}^j\}$ generated by the $j$th algorithm. If algorithm $j$ failed on a given problem instance (due to insufficient memory), we set $F_{j}^{\text{best}}=\infty$. Then, we compute the optimal objective function value gap as $(F_{j}^{\text{best}}-F^*)/{\max\{1, F^*\}}$ for all $j$. The results are visualized in \autoref{fig:Fval.diff}. The deeper the blue color of a rectangle for an algorithm, the better it performed in terms of achieving a lower objective value. On the flip side, the deeper the red color of a rectangle for an algorithm, the worse  it performed in terms of achieving a lower objective value. In Appendix~\ref{appendix:exp.performance}, we provide a discussion on the performance gap for the different methods.
\begin{figure}[!ht]
    \centering
    \includegraphics[width=0.5\textwidth]{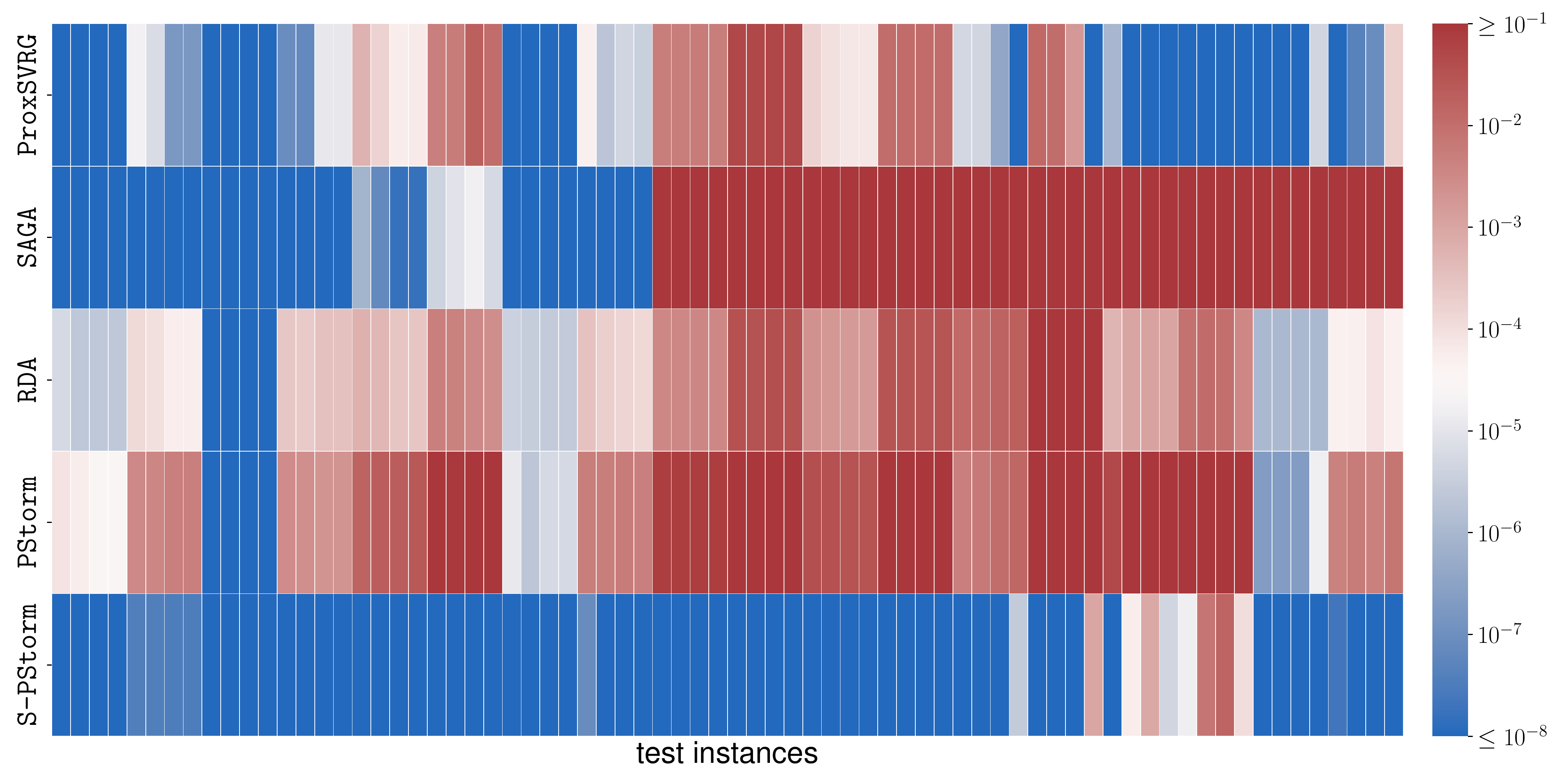}
    \caption{Visualization of objective value gaps for different methods. Each rectangular represents a test instance.
    \vspace{-15pt}
    }
    \label{fig:Fval.diff}
\end{figure}

Together \autoref{fig:suppot.id.performance} and \autoref{fig:Fval.diff} illustrate that \spstorm{} performed significantly better in both support identification and achieving better objective function values. 

Lastly, in Appendix~\ref{appendix:exp.additional}, we illustrate how the distance to the optimal solution $\norm{x_k-x^*}$ ($x^*$ is obtained using the \farsagroup{} algorithm) and error $\epsilon_k$ in the gradient estimator converge to $0$.  It can be observed empirically that the rates at which $\{\epsilon_k\}$ converges to $0$  and $\{x_k\}$ converges to $x^*$ agree with our $\Ocal(\sqrt{\log k/k})$ convergence result (see  Remark~\ref{rem:ek-error} and Remark~\ref{rem:xk_rate}).



\section{CONCLUSION}\label{sec.conclusion}

This paper proposes a new variance-reduced and stabilized stochastic proximal gradient method \spstorm{} for stochastic optimization with structured sparsity. Compared with existing methods, \spstorm{} has two new advantages. In terms of theoretical results, \spstorm{} has the consistent support identification property, which has not been proved for \rda{}. Regarding the efficiency and deployability, \spstorm{} neither requires any exact gradient evaluations nor needs to store a history of stochastic gradients. Numerical experiments on regularized logistic loss problems show that \spstorm{} outperforms popular methods in terms of both support identification and final objective function values obtained.

\noindent\textbf{Future directions.}
First, it would be interesting to investigate whether our consistent support identification results extend to the non-convex setting.  Second, our convergence and support identification results rely on exact evaluations of proximal operator, but some proximal operators, for example, overlapping group $\ell_1$ regularizers~\citep{obozinski2011group, yuan2013efficient}, do not admit closed-form solutions. We believe our results can be extended to this setting provided a subproblem solver is carefully designed to produce inexact proximal operator solutions geared towards support identification~\citep{dai2022inexact}.

\subsubsection*{Acknowledgements}
We thank the reviewers for their constructive comments that helped improve the paper. The authors Yutong Dai, Frank E. Curtis, and Daniel P. Robinson were supported by the US National Science Foundation grant DMS-2012243. The author Guanyi Wang was supported by the Singapore MOE under AcRF Tier-1 grant 22-5539-A0001.

\bibliographystyle{apalike}
\bibliography{reference}

\newpage
\appendix

\section{Proofs of Results in Section~\ref{sec.analysis}}

\subsection{Proof of Theorem~\ref{thm:error.in.gradient}} \label{app:error.in.gradient}
We first establish some useful lemmas. The first lemma establishes an upper bound on $ \left(\prod_{j= i}^{k}(1-\beta_j)\right)$, which will be used later to prove the variance reduction property.

\begin{lemma} \label{lemma:prod.1.minus.beta} 
  Under Assumption~\ref{ass:algo.choice} and with  $\underline{k}=\lceil (2c) - 1\rceil$, it holds for all $k\geq \underline{k}$ and $i\in \{2,3,\cdots,k\}$ that
$$
    \left(\prod_{j= i}^{k}(1-\beta_j)\right) \leq \exp\left( - \frac{\underline{k} - \min\{\underline{k}, i\}}{2} \right) \left( \frac{\max\{\underline{k}, i\} + 1}{k + 2}\right)^{c}.
$$
\end{lemma}
\begin{proof}
One can see from Assumption~\ref{ass:algo.choice} that
\begin{align*}
    \beta_j = \left\{
    \begin{array}{lll}
        \frac{1}{2} & \textup{if} ~ j < \underline{k} \\
        \frac{c}{j + 1} &  \textup{if} ~ j \geq \underline{k}.
    \end{array}
    \right.
\end{align*}
It follows from the above inequality and the fact that $1 - x \leq \exp(- x)$ for all $x \in \mathbb{R}$ that
\begin{align*}
    \left(\prod_{j= i}^{k}(1-\beta_j)\right) 
    \leq & ~ \exp\left(-\sum_{j=i}^{k}\beta_j\right)
    = 
    \begin{cases}
    \exp\left(-\sum_{j=i}^{k} \frac{c}{j+1}\right) & \text{ if } i \geq \underline{k},\\
    \exp\left(-\sum_{j=i}^{\underline{k}-1}\frac{1}{2}-\sum_{j=\underline{k}}^{k} \frac{c}{j+1}\right) & \text{ if } i < \underline{k},\\
    \end{cases}
    \\
    = & ~ \exp \left( - \frac{\underline{k} - \min\{\underline{k}, i\}}{2} - \sum_{j = \max\{\underline{k}, i\}}^k \frac{c}{j + 1} \right) \\ 
    \leq & ~ \exp\left( - \frac{\underline{k} - \min\{\underline{k}, i\}}{2} - \int_{x = \max\{\underline{k}, i\}}^{k+1} \frac{c}{x + 1} \mathrm{d}x \right) \\
    = & ~ \exp\left( - \frac{\underline{k} - \min\{\underline{k}, i\}}{2} \right) \left( \frac{\max\{\underline{k}, i\} + 1}{k + 2}\right)^{c},
\end{align*}
where the second inequality follows from $\int_{a}^{b+1} \frac{1}{x} d x<\sum_{j=a}^b \frac{1}{i}$ for any $0<a\leq b$. This completes the proof. 
\end{proof}

The next lemma establishes, for all $k$, a relationship between the stochastic gradient error $\epsilon_k=d_k-\grad f(x_k)$ and a martingale. This is useful for an Azuma-Hoeffding-type inequality that will be used to prove a variance reduction property.

\begin{lemma}\label{lemma:giant}
For all $k\geq 2$, with the convention that $\prod_{i=l}^{u}a_i=1$ if $l>u$, consider $\{e_{ki}\}_{i=0}^{k}$ with
\begin{align*}
    e_{ki} := 
    \begin{cases}
    0 & i = 0, \\
    \left(\prod_{j=2}^{k}(1-\beta_j)\right) A_1 & i = 1,\\
    \left(\prod_{j=i+1}^{k}(1-\beta_j)\right) A_i + \left(\prod_{j=i}^{k}(1-\beta_j)\right)B_i & 2\leq i\leq k,
    \end{cases}
\end{align*}
where $A_i := v_i - \grad f(x_i)$ and $B_{i} := \grad f(x_{i-1}) - u_{i}$ for all $i\geq 1$ with $v_i$ and $u_i$ defined as in Algorithm~\ref{algo:main}.
\begin{enumerate}
    \item Consider $\{S_{kt}\}_{t=0}^\infty$ with $S_{kt} := \sum_{i=0}^t e_{ki}$ for all $0 \leq t \leq k$ and $S_{kt} = S_{kk}$ for all $t > k$. Under Assumption~\ref{ass:proxstorm.varinace.reduced}(\ref{ass:proxstorm.varinace.reduced.1}), $\{S_{kt}\}_{t=0}^\infty$ forms a martingale with respect to the filtration $\{\Fcal_{t}\}_{t=0}^\infty$. Specifically, with $\Fcal_0 = \Fcal_1 = \sigma(x_1)$ and $\Fcal_t$ is the $\sigma$-algebra generated by $\{\{\Xi_{1,i}\}_{i=1}^{m}, \dots, \{\Xi_{(t-1),i}\}_{i=1}^{m}\}$ (of which $\{\{\xi_{1,i}\}_{i=1}^{m}, \cdots, \{\xi_{(t-1),i}\}_{i=1}^{m}\}$ is a realization) for all $t\in\{2,\cdots,k\}$, and $\Fcal_t = \Fcal_k$ for all $t>k$.
    \label{lemma:grad.error.martingale}
    
    \item With $\{S_{kt}\}_{t=0}^\infty$ defined as in part \ref{lemma:grad.error.martingale}, one has that $S_{kk}=\epsilon_{k}$. \label{lemma:grad.error.decomposition}
    
    \item\label{lemma:eki.norm.bound} Under Assumption~\ref{ass:proxstorm.varinace.reduced} and Assumption~\ref{ass:algo.choice} and with $\underline{k}=\lceil (2c) - 1\rceil$, it holds almost surely that
    \begin{align*}
    \norm{e_{ki}} \leq
    \begin{cases}
    \sigma \exp\left( - \frac{\underline{k} - 2}{2} \right) \left( \frac{\underline{k} + 1}{k + 2}\right)^{c} & \textup{if} ~~ i = 1,  \\
        \big(2 \sigma + 2L_g  (G_r+G_d)\zeta \ualpha \big) \frac{1}{2} \exp\left( - \frac{\underline{k} - i}{2} \right) \left( \frac{\underline{k} + 1}{k + 2}\right)^{c} & \textup{if} ~~ 2 \leq i \leq \underline{k}, \\
        \big(2 \sigma + 2L_g  (G_r+G_d)\zeta \ualpha \big) \frac{c}{i} \left( \frac{i + 1}{k + 2}\right)^{c} & \textup{if} ~~ \underline{k} + 1 \leq i \leq k. 
    \end{cases}
    \end{align*}
\end{enumerate}
\end{lemma}
\begin{proof}
  Consider part~\ref{lemma:grad.error.martingale}.  We have  $S_{k0} = e_{k0} = 0$, and for all $1 \leq t \leq k$, one finds $S_{kt}-S_{k(t-1)}=e_{kt}$, so that
\begin{equation}\label{eq:martingale}
\E{\xi\sim\Pcal}{S_{kt} | \Fcal_t} = \E{\xi\sim\Pcal}{S_{k(t-1)} + e_{kt} | \Fcal_t} = S_{k(t-1)} + \E{\xi\sim\Pcal}{e_{kt} | \Fcal_t}.
\end{equation} 
Assumption~\ref{ass:proxstorm.varinace.reduced}(\ref{ass:proxstorm.varinace.reduced.1}) implies that $\E{\xi\sim\Pcal}{e_{kt} | \Fcal_t}=0$, which may then be combined with \eqref{eq:martingale} to conclude that $\E{\xi\sim\Pcal}{S_{kt} | \Fcal_t} = S_{k(t-1)}$ for all $1 \leq t \leq k$.  On the other hand, for all $t > k$, we trivially have $\E{\xi\sim\Pcal}{S_{kt} | \Fcal_t} = \E{\xi\sim\Pcal}{S_{k(t-1)}| \Fcal_t} = S_{k(t-1)}$.  Therefore,  $\{S_{kt}\}_{t=0}^\infty$ forms a martingale. 

Consider part~\ref{lemma:grad.error.decomposition}. For all $k\geq 2$, one finds that
\begin{align*}
\epsilon_k & = d_k-\grad f(x_k)    \\
& = (1-\beta_k)\epsilon_{k-1} + A_k + (1-\beta_k)B_{k} \\
& = (1-\beta_k)(1-\beta_{k-1})\epsilon_{k-2} + (1-\beta_k)A_{k-1} + A_k + (1-\beta_k) (1-\beta_{k-1})B_{k-1} + (1-\beta_k)B_{k}\\
& = \left(\prod_{j=2}^{k}(1-\beta_j)\right)\epsilon_{1} + \sum_{i=2}^{k}\left(\prod_{j=i+1}^{k}(1-\beta_j)\right)A_i + \sum_{i=2}^{k}\left(\prod_{j=i}^{k}(1-\beta_j)\right)B_i.
\end{align*}
Since $\epsilon_1 = A_1$, the desired conclusion follows that $\epsilon_k = \sum_{i=0}^k e_{ki} = S_{kk}$.

We now prove part~\ref{lemma:eki.norm.bound}.
Consider the following two cases: 
\paragraph{Case I:}For $i = 1$, it follows from the triangular inequality and Assumption~\ref{ass:proxstorm.varinace.reduced}(\ref{ass:proxstorm.varinace.reduced.4}) that, almost surely, one finds
\begin{align*}
    \norm{e_{k1}}
    &=\norm{\left(\prod_{j=2}^{k}(1-\beta_j)\right)\epsilon_{1}}\leq \left(\prod_{j=2}^{k}(1-\beta_j)\right)\norm{\epsilon_1}\\
    &= \left(\prod_{j=2}^{k}(1-\beta_j)\right)\norm{\frac{1}{m}\sum_{i'=1}^m\grad\ell(x_1;\xi_{1i'})-\grad f(x_1)}\\
    &\leq \sigma \left(\prod_{j=2}^{k}(1-\beta_j)\right). 
\end{align*}
It follows from Lemma~\ref{lemma:prod.1.minus.beta} that, almost surely, one finds
\begin{align*}
    \norm{e_{k1}} \leq \sigma \exp\left( - \frac{\underline{k} - 2}{2} \right) \left( \frac{\underline{k} + 1}{k + 2}\right)^{c}. 
\end{align*}

\paragraph{Case II:} For any $i$ with $2 \leq i \leq k$, it follows almost surely that
\begin{align}
&\phantom{iii} \norm{e_{ki}} \\
& =\norm{\left(\prod_{j=i+1}^{k}(1-\beta_j)\right)A_i + \left(\prod_{j=i}^{k}(1-\beta_j)\right)B_i}\nonumber\\
& =\norm{\left(\prod_{j=i+1}^{k}(1-\beta_j)\right)(1-\beta_i+\beta_i)A_i + \left(\prod_{j=i}^{k}(1-\beta_j)\right)B_i}\nonumber\\
& = \norm{\beta_i\left(\prod_{j=i+1}^{k}(1-\beta_j)\right)A_i + \left(\prod_{j=i}^{k}(1-\beta_j)\right)(A_i+B_i)}\nonumber\\
& \leq \sigma\beta_i\left(\prod_{j=i+1}^{k}(1-\beta_j)\right) + \left(\prod_{j=i}^{k}(1-\beta_j)\right)\norm{\frac{1}{m}\sum_{i'=1}^{m}\grad \ell(x_i;\xi_{ii'})- \frac{1}{m}\sum_{i'=1}^{m}\grad \ell(x_{i-1};\xi_{ii'}) - (\grad f(x_i)- \grad f(x_{i-1}))}\nonumber\\
& \leq \sigma\beta_i\left(\prod_{j=i+1}^{k}(1-\beta_j)\right) + 2L_g\left(\prod_{j=i}^{k}(1-\beta_j)\right)\norm{x_{i}-x_{i-1}},\label{eq:bound.eki.1}
\end{align}
where the first inequality holds by Assumption~\ref{ass:proxstorm.varinace.reduced}(\ref{ass:proxstorm.varinace.reduced.4}) and the second inequality holds by Assumption~\ref{ass:generic.assumptions}(\ref{ass:generic.assumptions.2}). Since $y_{i-1}=\prox{\alpha_{i-1}r}{x_{i-1}-\alpha_{i-1}d_{i-1}}$, it follows from \citet[Theorem 6.39]{beck2017first} that $\frac{x_{i-1}-y_{i-1}}{\alpha_{i-1}}-d_{i-1}\in\partial r(y_{i-1})$.  Hence, it follows from  Assumption~\ref{ass:proxstorm.varinace.reduced}(\ref{ass:proxstorm.varinace.reduced.3}) that  $\norm{\frac{x_{i-1}-y_{i-1}}{\alpha_{i-1}}-d_{i-1}}\leq G_r$. It follows from line~\ref{line:stablelization} of Algorithm~\ref{algo:main},  Assumption~\ref{ass:proxstorm.varinace.reduced}(\ref{ass:proxstorm.varinace.reduced.5}), the triangular inequality, and the previous inequality that
$$
\begin{aligned}
 \norm{x_{i} - x_{i-1}} 
 & = \zeta \beta_{i-1}\norm{x_{i-1}-y_{i-1}} \\
 & \leq \zeta \beta_{i-1}(\norm{x_{i-1}-y_{i-1}-\alpha_{i-1}d_{i-1}} + \alpha_{i-1}\norm{d_{i-1}})\\
 &  \leq \zeta \beta_{i-1}\alpha_{i-1}(G_r+G_d).
\end{aligned}
$$
Combining \eqref{eq:bound.eki.1} and the above inequality, one finds almost surely that
\begin{equation}\label{eq:bound.eki.2}
\begin{aligned}
\norm{e_{ki}} 
& \leq \sigma\beta_i\left(\prod_{j=i+1}^{k}(1-\beta_j)\right) + 2L_g(G_r+G_d)\zeta \beta_{i-1}\alpha_{i-1}\left(\prod_{j=i}^{k}(1-\beta_j)\right).
\end{aligned}   
\end{equation}
It follows from Assumption~\ref{ass:algo.choice} that $\beta_k = \min\{ \frac{1}{2}, ~  \frac{c}{(k+1)} \}$ for all $k\geq 2$. Therefore, since $2 (1 - \beta_i) \geq 1$, one finds
\begin{equation}\label{eq:bound.eki.3}
    \beta_{i}\left(\prod_{j=i+1}^{k}(1-\beta_j)\right)
    \leq 2 \beta_{i} \left(\prod_{j=i}^{k}(1-\beta_j)\right) \leq 2 \beta_{i - 1} \left(\prod_{j=i}^{k}(1-\beta_j)\right) ,
\end{equation}
It follows from \eqref{eq:bound.eki.2}, \eqref{eq:bound.eki.3}, and $\alpha_i\equiv \ualpha$ that almost surely one finds
$$
\norm{e_{ki}}\leq \big(2 \sigma + 2L_g  (G_r+G_d)\zeta \ualpha \big) \beta_{i-1}\left(\prod_{j=i}^{k}(1-\beta_j)\right).
$$
Applying Lemma~\ref{lemma:prod.1.minus.beta} to the above inequality, one finds almost surely that
\begin{align*}
    \norm{e_{ki}} \leq \left\{
    \begin{array}{lll}
        \big(2 \sigma + 2L_g  (G_r+G_d)\zeta \ualpha \big) \frac{1}{2} \exp\left( - \frac{\underline{k} - i}{2} \right) \left( \frac{\underline{k} + 1}{k + 2}\right)^{c} & \textup{if} ~~ 2 \leq i \leq \underline{k}, \\
        \big(2 \sigma + 2L_g  (G_r+G_d)\zeta \ualpha \big) \frac{c}{i} \left( \frac{i + 1}{k + 2}\right)^{c} & \textup{if} ~~ \underline{k} + 1 \leq i \leq k.  \\
    \end{array}
    \right. 
\end{align*}
Combining the two cases above give the results claimed in part~\ref{lemma:eki.norm.bound}.
\end{proof}

The last lemma bounds $\sum_{i=1}^k \norm{e_{ki}}^2$, which will appear in the Azuma-Hoeffding type inequality.
\begin{lemma} \label{lemma:sum.eki.norm}
Under Assumption~\ref{ass:proxstorm.varinace.reduced} and Assumption~\ref{ass:algo.choice}, there exits a constant $C_1>0$ that is independent of $k$ such that, for all $k\geq\underline{k}=\lceil (2c) - 1\rceil$, one finds
\begin{align*}
    \sum_{i=1}^k \norm{e_{ki}}^2  \leq C_1 \big( \sigma + L_g  (G_r+G_d)\zeta \ualpha \big)^2 \max\left\{ \left( \frac{\underline{k} + 1}{k + 2}\right)^{2c}, ~ \frac{c^2}{k+2} \right\} \ \ \text{almost surely}.
\end{align*}
\end{lemma}

\begin{proof}
It follows from Lemma~\ref{lemma:giant}(\ref{lemma:eki.norm.bound}) that, almost surely,
\begin{align}
    \sum_{i=1}^k \norm{e_{ki}}^2 
    = & ~  \norm{e_{k1}}^2 + \sum_{i=2}^{\underline{k}} \norm{e_{ki}}^2 + \sum_{i= \underline{k} + 1}^{k} \norm{e_{ki}}^2 \nonumber\\
    \leq & ~ \sigma^2 \exp\left( - (\underline{k} - 2) \right) \left( \frac{\underline{k} + 1}{k + 2}\right)^{2c} + \sum_{i = 2}^{\underline{k}} \big(2 \sigma + 2L_g  (G_r+G_d)\zeta \ualpha \big)^2 \frac{1}{4} \exp\left( - (\underline{k} - i) \right) \left( \frac{\underline{k} + 1}{k + 2}\right)^{2c} \nonumber\\
    & ~ + \sum_{i= \underline{k} + 1}^{k} \big(2 \sigma + 2L_g  (G_r+G_d)\zeta \ualpha \big)^2 \frac{c^2}{i^2} \left( \frac{i + 1}{k + 2}\right)^{2c}.\label{eq:sum.eki.norm.1}
\end{align}
With respect to each of three terms above, for some $C_{11}$ that is independent of $k$, one finds
\begin{align}
    & \sigma^2 \exp\left( - (\underline{k} - 2) \right) \cdot \left( \frac{\underline{k} + 1}{k + 2}\right)^{2c} =  \sigma^2 e^2 \exp(- \underline{k}) \cdot \left( \frac{\underline{k} + 1}{k + 2}\right)^{2c}\label{eq:sum.eki.norm.2} \\
    & \sum_{i = 2}^{\underline{k}} \big(2 \sigma + 2L_g  (G_r+G_d)\zeta \ualpha \big)^2 \frac{1}{4} \exp\left( - (\underline{k} - i) \right) \cdot \left( \frac{\underline{k} + 1}{k + 2}\right)^{2c} \leq  \big( \sigma + L_g  (G_r+G_d)\zeta \ualpha \big)^2 \frac{e}{e - 1} \cdot \left( \frac{\underline{k} + 1}{k + 2}\right)^{2c}\label{eq:sum.eki.norm.3} \\
    & \sum_{i= \underline{k} + 1}^{k} \big(2 \sigma + 2L_g  (G_r+G_d)\zeta \ualpha \big)^2 \frac{c^2}{i^2} \cdot \left( \frac{i + 1}{k + 2}\right)^{2c} 
    \leq   \frac{\big(2 \sigma + 2L_g  (G_r+G_d)\zeta \ualpha \big)^2 c^2 C_{11}}{k+2},
    \label{eq:sum.eki.norm.4}
\end{align}
where \eqref{eq:sum.eki.norm.3} holds since the geometric series $\sum_{i = 2}^{\underline{k}}  \exp\left( - (\underline{k} - i) \right) = \sum_{i=2}^{k}\frac{\exp(i)}{\exp(k)} = \frac{ e - e^{2 - \underline{k}}}{e - 1} \leq \frac{e}{e - 1}$ and \eqref{eq:sum.eki.norm.4} hold since 
\begin{align}
    \sum_{i = 1}^k \frac{(i + 1)^{2c}}{i^2} = & ~ 
    \sum_{i = 1}^{1} \frac{(i + 1)^{2c}}{i^2} + \sum_{i = 2}^{k} \frac{(i + 1)^{2c}}{i^2} \nonumber \\
    \leq & ~ 4^c + \sum_{i = 2}^{k} \frac{(1.5 i)^{2c}}{i^2} \nonumber \\
    \leq & ~ 4^c + (1.5)^{2c} \int_{i = 2}^{k + 1} i^{2c - 2} \mathrm{d}i \nonumber \\
    = & ~ 4^c + (1.5)^{2c} \left( \frac{(k + 1)^{2c - 1}}{2c - 1} - \frac{2^{2c - 1}}{2c - 1} \right) 
    \leq C_{11} (k+1)^{2c - 1} \leq C_{11} (k+2)^{2c - 1}.
    \label{eq:sum-of-polynomial}
\end{align}
Combining \eqref{eq:sum.eki.norm.1}-\eqref{eq:sum.eki.norm.4}, one finds almost surely that 
\begin{align*}
    \sum_{i=1}^k \norm{e_{ki}}^2  
    & \leq  \big( \sigma + L_g  (G_r+G_d)\zeta \ualpha \big)^2 \left( C_{12}\left( \frac{\underline{k} + 1}{k + 2}\right)^{2c} + \left(4C_{11}+\frac{e}{e-1}\right) \frac{c^2}{k+2} \right) \\
    & \leq \big( \sigma + L_g  (G_r+G_d)\zeta \ualpha \big)^2 \left(C_{12} + 4C_{11}+\frac{e}{e-1}\right)  \max\left\{ \left( \frac{\underline{k} + 1}{k + 2}\right)^{2c}, ~ \frac{c^2}{k+2} \right\},
\end{align*} 
where we use the fact that $\sigma^2 e^2 \exp(- \underline{k}) \leq C_{12} \big( \sigma + L_g  (G_r+G_d)\zeta \ualpha \big)^2$ for some $C_{12}>0$ that is independent of $k$. We complete the proof by setting $C_1=\left(C_{12} + 4C_{11}+\frac{e}{e-1}\right)$.
\end{proof}
Now, we are ready to formally prove Theorem~\ref{thm:error.in.gradient}.
\paragraph{Theorem~\ref{thm:error.in.gradient}.} 
Let Assumption~\ref{ass:generic.assumptions}--Assumption~\ref{ass:algo.choice} hold, let $\epsilon_k=d_k-\grad f(x_k)$ for all $k \in \N{}_{+}$, and define $\underline{k}=\lceil (2c) - 1\rceil$.  Then, for any $k \geq \underline{k}$ and any $\eta_k \in (0,1)$, the event $\Ecal_k := \{\norm{\epsilon_k} \leq U(k)\}$ holds with probability at least $1 - \eta_k$, where for some constant $C$ independent of $k$, one defines
\begin{align*}
   U(k) = C \big( \sigma + L_g  (G_r+G_d) \zeta \ualpha \big) \cdot \max\left\{ \left( \frac{\underline{k} + 1 }{k + 2}\right)^{c}, ~ \frac{c}{\sqrt{k+2}} \right\} \sqrt{\log\frac{2}{\eta_k}}.
\end{align*}
(Specifically, the constant is $C = \sqrt{2C_1}$, where $C_1$ is defined in Lemma~\ref{lemma:sum.eki.norm}.)

\begin{proof}
It follows from Lemma~\ref{lemma:sum.eki.norm} that almost surely one finds
$$
\sum_{i=1}^k \norm{e_{ki}}^2  \leq C_1 \big( \sigma + L_g  (G_r+G_d)\zeta \ualpha \big)^2 \max\left\{ \left( \frac{\underline{k} + 1}{k + 2}\right)^{2c}, ~ \frac{c^2}{k+2} \right\} =: h(k).
$$

Based on Lemma~\ref{lemma:giant}(\ref{lemma:grad.error.martingale}), we have for $k \geq \underline{k}$ that $\{S_{kt}\}_{t=0}^k$ forms a martingale with respect to the filtration $\{\Fcal_{t}\}_{t = 0}^{k}$. For any $\rho_k>0$, using the Azuma-Hoeffding type inequality~\citep[\href{https://arxiv.org/pdf/1208.2200.pdf}{Theorem 3.5}]{pinelis1994optimum}\footnote{See Remark~\ref{remark:apply.ah.thm} for details on applying this theorem.} on the martingale $\{S_{kt}\}_{t = 0}^k$,  together with $\norm{e_{ki}}_{\infty}\leq \norm{e_{ki}}$ ($e_{ki}$ is defined in  Lemma~\ref{lemma:giant}) and the fact that $S_{kk}=\epsilon_k$ (Lemma~\ref{lemma:giant}(\ref{lemma:grad.error.decomposition})), we have 
\begin{equation}\label{eq:azuma}
    \P{\norm{\epsilon_k}\geq \rho_k} = \P{\norm{S_{kk}}\geq \rho_k} \leq \P{\sup_{t\in[k]}\norm{S_{kt}}\geq \rho_k} 
    \leq 2\exp\left(-\frac{\rho_k^2}{2h(k)}\right).
\end{equation}

For any $\eta_k \in (0,1)$, by setting $\rho_k = U(k) = \sqrt{2 h(k) \log(2 / \eta_k)}$ in~\eqref{eq:azuma}, we have $\mathbb{P}\big[ \norm{\epsilon_k} \geq U(k) \big] \leq \eta_k$, which implies that the event  $\mathcal{E}_k = \{\norm{\epsilon_k} \leq U(k)\}$ holds with probability at least $1 - \eta_k$. This completes the proof. 
\end{proof}

\begin{remark}\label{remark:apply.ah.thm}
We define the $f$ used in~\citep[Theorem~3.5]{pinelis1994optimum} when cited in the proof of Theorem~\ref{thm:error.in.gradient} above as  $f=\{S_{k0}, S_{k1},\cdots,S_{kk},S_{kk},\dots\}$ with $f_j=S_{kj}$ for all $1\leq j\leq k$ and $f_j = S_{kk}$ for all $j > k$. As proved in Lemma~\ref{lemma:giant}(\ref{lemma:grad.error.martingale}), $f$ is a martingale. 
Consequently, the $d_j$ and $f^*$ appearing in \citep[Theorem~3.5]{pinelis1994optimum} are defined as $d_j=S_{kj}-S_{k(j-1)}=e_{kj}$ and $f^*=\sup_{j\in[k]}\{\norm{f_j}\}=\sup_{j\in[k]}\{\norm{S_{kj}}\}$, respectively. As proved in Lemma~\ref{lemma:sum.eki.norm}, we have $\sum_{j=1}^{k}\norm{d_j}^2_{\infty}\leq \sum_{j=1}^{k}\norm{d_j}^2_{2}=\sum_{j=1}^{k}\norm{e_{kj}}^2_{2}\leq h(k)$ almost surely. 
\end{remark}

\subsection{Proof of Corollary~\ref{cor:variance.all.iters}}\label{app:proof.for.corllary}
\paragraph{Corollary~\ref{cor:variance.all.iters}}Let $\eta_k =\frac{\eta_0}{k^2}$ for all $k\geq 1$ with $\eta_0 \in (0, 6 / \pi^2)$. Define the event $\Ecal_k := \{\norm{\epsilon_k} \leq U(k)\}$ and recall that $\underline{k}=\lceil (2c) - 1\rceil$. Under Assumption~\ref{ass:generic.assumptions}--Assumption~\ref{ass:algo.choice}, the event $\Ecal := \bigcap_{k\geq\underline{k}}^{\infty} \mathcal{E}_k$ holds with probability at least $1 - \frac{\eta_0 \pi^2}{6}$.
\begin{proof}
It follows from the stated conditions, the union bound from probability, and Theorem~\ref{thm:error.in.gradient} that
$$
\begin{aligned}
\P{\bigcap_{k=\underline{k}}^{\infty}\left\{\norm{\epsilon_k} \leq U(k)\right\}}
& = \P{\bigcap_{k=\underline{k}}^{\infty}\Ecal_k} = 1 - \P{\left(\bigcap_{k=\underline{k}}^{\infty}\Ecal_k\right)^c} \qquad(\text{here $c$ is the set complement operator})\\
& = 1 - \P{\bigcup_{k=\underline{k}}^{\infty}\Ecal_k^c}\geq 1 - \sum_{k\geq \underline{k}}^{\infty}\P{\Ecal_k^c}
= 1 - \sum_{k\geq \underline{k}}^{\infty}\P{\norm{\epsilon_k} > U(k)}\\
& \geq 1- \sum_{k\geq \underline{k}}^{\infty}\eta_k \geq 1-\sum_{k=1}^{\infty}\frac{\eta_0}{k^2} =  1 - \frac{\eta_0\pi^2}{6},
\end{aligned}
$$
where the last equality holds by the Basel equality $\sum_{k=1}^{\infty} \frac{1}{k^2}=\frac{\pi^2}{6}$.
\end{proof}
\subsection{Proof of Theorem~\ref{thm:iterate.complexity}} \label{app:iterate.complexity}

\paragraph{Theorem~\ref{thm:iterate.complexity}.} 
Let $\ualpha = \mu_f / L_g^2$, $\zeta \in (0, 2)$, $\theta\geq 2$, $c = (2\theta L_g^2) / (\zeta \mu_f^2) > 2$,  and $\underline{k}=\lceil 2c-1\rceil$. Set $\eta_k = \eta_0 / k^2$ for all $k\geq 1$ with $\eta_0 \in (0, 6 / \pi^2)$. Then, under Assumption~\ref{ass:generic.assumptions}--Assumption~\ref{ass:algo.choice}, there exists a constant $C_3>0$ independent of $k$, such that the event $\Ecal^x_{k}:=\left\{\norm{x_{k}-x^*}^2 \leq {\bar c}_1 \frac{\norm{x_{\underline{k}}-x^*}^2}{k^\theta} + {\bar c}_2\cdot\frac{\log\frac{2k}{\eta_0}}{k}\right\}$ with ${\bar c}_1:=(\underline{k} + 2)^\theta$ and ${\bar c}_2:=C_3\zeta\left(\frac{\mu_f^2}{L_g^4} + \frac{2}{L_g^2}\left(1+\frac{\mu_f}{L_g}\right)^2\right)(\sigma + L_g(G_r + G_d)\zeta \ualpha)^2$ satisfies
$$
\P{\bigcap_{k\geq\underline{k}}^{\infty}\Ecal^x_{k}} \geq 1 - \eta_0 \pi^2 / 6 > 0 \, .
$$

\begin{proof}
Since the proximal operator is non-expansive~\citep[Theorem 6.42]{beck2017first} and $x^*=\prox{\alpha_kr}{x^*-\alpha_k\grad f(x^*)}$, it follows that
\begin{align}
   \norm{y_{k}-x^*}^2
   & = \norm{\prox{\alpha_kr}{x_k-\alpha_kd_k} - \prox{\alpha_kr}{x^*-\alpha_k\grad f(x^*)}}^2\nonumber\\
   & \leq \norm{x_k-x^*-\alpha_k(d_k -\grad f(x^*))}^2\nonumber\\
   & = \norm{x_k-x^*}^2 - 2\alpha_k(x_k-x^*)^T(d_k -\grad f(x^*)) + \alpha_k^2\norm{d_k -\grad f(x^*)}^2\nonumber\\
   & = \norm{x_k-x^*}^2 - 2\alpha_k(x_k-x^*)^T(\epsilon_k + \grad f(x_k) -\grad f(x^*)) + \alpha_k^2\norm{d_k -\grad f(x^*)}^2. \label{eq:proxstorm.iter.expand1}
\end{align}
It follows from Assumption~\ref{ass:generic.assumptions} that $f$ is $\mu_f$-strongly convex, and therefore
\begin{equation}\label{eq:proxstorm.iter.expand2.pre}
    (x_k-x^*)^T(\grad f(x_k) -\grad f(x^*))\geq \mu_f\norm{x_k-x^*}^2.
\end{equation}
It follows from \eqref{eq:proxstorm.iter.expand1} that
\begin{align}
  &\norm{y_{k}-x^*}^2 \\
  & \leq \norm{x_k-x^*}^2 -  2\alpha_k(x_k-x^*)^T(\grad f(x_k) -\grad f(x^*)) - 2\alpha_k(x_k-x^*)^T\epsilon_k + \alpha_k^2\norm{d_k -\grad f(x^*)}^2 \nonumber\\
  & \overset{(i)}{\leq} (1- 2\mu_f\alpha_k)\norm{x_k-x^*}^2 - 2\alpha_k(x_k-x^*)^T\epsilon_k + \alpha_k^2\norm{d_k -\grad f(x^*)}^2\nonumber\\
  & =(1- 2\mu_f\alpha_k)\norm{x_k-x^*}^2 + 2\alpha_k(x^*-x_k)^T\epsilon_k + \alpha_k^2\norm{d_k -\grad f(x_k)+ \grad f(x_k)-\grad f(x^*)}^2\nonumber\\
  & \overset{(ii)}{=}(1- 2\mu_f\alpha_k)\norm{x_k-x^*}^2 + 2\alpha_k(x^*-x_k)^T\epsilon_k + \alpha_k^2\left(\norm{\epsilon_k}^2+ 2\epsilon_k^T(\grad f(x_k)-\grad f(x^*)) +\norm{\grad f(x_k)-\grad f(x^*)}^2\right)\nonumber\\
  &\overset{(iii)}{\leq}(1- 2\mu_f\alpha_k)\norm{x_k-x^*}^2 + 2\alpha_k\norm{x_k-x^*}\norm{\epsilon_k} + \alpha_k^2\left(\norm{\epsilon_k}^2 + 2L_g\norm{\epsilon_k}\norm{x_k-x^*}+L_g^2\norm{x_k-x^*}^2\right)\nonumber\\
  & = (1-2\mu_f\alpha_k+\alpha_k^2L_g^2)\norm{x_k-x^*}^2 + (2\alpha_k + 2L_g\alpha_k^2)\norm{x_k-x^*}\norm{\epsilon_k} + \alpha_k^2\norm{\epsilon_k}^2,
  \label{eq:proxstorm.iter.expand2}
\end{align}
where $(i)$ follows from~\eqref{eq:proxstorm.iter.expand2.pre}, $(ii)$ follows from the definition of $\epsilon_k$, and $(iii)$ follows from Assumption~\ref{ass:generic.assumptions} and the Cauchy-Schwarz inequality.
When the event $\mathcal{E}_k = \{\norm{\epsilon_k} \leq U(k)\}$ happens ($U(k)$ defined in Theorem~\ref{thm:error.in.gradient}), it follows from line~\ref{line:stablelization} in Algorithm~\ref{algo:main}, $\zeta\beta_k< 1$, \eqref{eq:proxstorm.iter.expand2}, and Theorem~\ref{thm:error.in.gradient} that
\begin{align}
  &\norm{x_{k+1}-x^*}^2\nonumber\\
  & = \norm{\zeta \beta_k(y_{k}-x^*) + (1-\zeta \beta_k)(x_k-x^*)}^2  \nonumber\\
  & \leq \zeta \beta_k\norm{y_k-x^*}^2 + (1- \zeta \beta_k)\norm{x_k-x^*}^2 \nonumber\\
  & \leq \zeta \beta_k\left((1-2\mu_f\alpha_k+\alpha_k^2L_g^2)\norm{x_k-x^*}^2 + (2\alpha_k + 2L_g\alpha_k^2)\norm{x_k-x^*}\norm{\epsilon_k} + \alpha_k^2\norm{\epsilon_k}^2\right) + (1- \zeta \beta_k)\norm{x_k-x^*}^2 \nonumber\\
  &\leq (1- \zeta \beta_k(2\mu_f\alpha_k-\alpha_k^2L_g^2))\norm{x_k-x^*}^2 + (2\alpha_k + 2L_g\alpha_k^2)\norm{x_k-x^*} \zeta \beta_kU(k) + \alpha_k^2 \zeta \beta_kU(k)^2; 
  \label{eq:proxstorm.iter.expand3} 
\end{align}
we emphasize that the first inequality above follows from the convexity of the 2-norm-squared. Therefore, \eqref{eq:proxstorm.iter.expand3} holds with probability at least $1-\eta_k$ since  the event $\mathcal{E}_k = \{\norm{\epsilon_k} \leq U(k)\}$ happens with probability at least $1-\eta_k$.

Define $s_k^2 = \norm{x_k-x^*}^2$ and since $\alpha_k \equiv \ualpha = \mu_f / L_g^2$, then \eqref{eq:proxstorm.iter.expand3} becomes
\begin{align}
    s_{k + 1}^2
    & \leq \left(1- \zeta \beta_k\frac{\mu_f^2}{L_g^2}\right)s_k^2 + \frac{2 \mu_f}{L_g^2} \left(1 + \frac{\mu_f}{L_g} \right)\zeta \beta_k U(k) s_k + \frac{\mu_f^2}{L_g^4}\zeta \beta_kU(k)^2. \nonumber\\
    &= (1 - c_0 \zeta \beta_k) s_k^2 + c_1 \zeta \beta_k U(k) s_k + c_2\zeta \beta_k U(k)^2,\label{eq:eq:proxstorm.iter.expand4}
\end{align}
with $c_0 = \frac{\mu_f^2}{L_g^2}, c_1 = \frac{2 \mu_f}{L_g^2} \left(1 + \frac{\mu_f}{L_g} \right)$, and $c_2 = \frac{\mu_f^2}{L_g^4}$.
The second term in the above inequality can be upper bounded as
\begin{align*}
    c_1 \zeta \beta_k U(k) s_k= 2 \left(\frac{c_1\sqrt{\rho\zeta\beta_k}}{2}s_k\right) \left(\frac{\sqrt{\zeta\beta_k}}{\sqrt{\rho}} U(k)\right) \leq \frac{\rho \zeta\beta_kc_1^2}{4}  s_k^2 + \frac{\zeta \beta_k}{\rho} U(k)^2 \;\;\text{ for all } \rho > 0, 
\end{align*}
by using Young's inequality. Combining this result with \eqref{eq:eq:proxstorm.iter.expand4}, one obtains
\begin{align*}
    & ~ s_{k + 1}^2 \leq \left[1 - \left(c_0 \zeta - \rho \zeta \frac{c_1^2}{4}\right) \beta_k\right] s_k^2 + \left[c_2 \zeta + \frac{\zeta}{\rho}\right] \beta_k U(k)^2. 
\end{align*}
Now setting $\rho = \frac{2 \mu_f^2}{ L_g^2 c_1^2}$, it follows from this inequality that
\begin{align}\label{eq:proxstorm.iter.expand5}
    & s_{k + 1}^2 \leq \left[ 1 - \frac{\zeta \mu_f^2}{2 L_g^2} \beta_k \right] s_k^2 + \left[ \frac{\zeta \mu_f^2}{L_g^4} + \frac{\zeta c_1^2 L_g^2}{2 \mu_f^2} \right]  \beta_k U(k)^2 = (1-\gamma_k)s_k^2 + c_3\beta_k U(k)^2,
\end{align}
where $\gamma_k=\frac{\zeta \mu_f^2}{2 L_g^2} \beta_k$ and $c_3=\frac{\zeta \mu_f^2}{L_g^4} + \frac{\zeta c_1^2 L_g^2}{2 \mu_f^2}$.

{\bf{Conditioning on the event $\mathcal{E} = \bigcap_{i\geq \underline{k}}^{\infty} \mathcal{E}_i$ happens}}, it follows from~\eqref{eq:proxstorm.iter.expand5}, for all $k\geq \underline{k}$, that
\begin{align}\label{eq:proxstorm.iter.expand6}
    s_{k + 1}^2 
    &\leq  (1 - \gamma_k) s_k^2 + c_3 \beta_k U(k)^2\nonumber \\
    & \leq (1 - \gamma_k)(1 - \gamma_{k-1})s_{k-1}^2 + c_3 \sum_{i=k-1}^k\left(\prod_{j=i+1}^{k}(1-\gamma_j)\beta_iU(i)^2\right)\nonumber\\
    &\leq  \text{(expanding recursively on $s_{k-1}$)}\nonumber\\
    & \leq \left[ \prod_{i = \underline{k}}^k (1 - \gamma_i) \right] \cdot s_{\underline{k}}^2 + c_3 \cdot \sum_{i = \underline{k}}^k \left[ \prod_{j = i + 1}^k (1 - \gamma_j) \right] \beta_i U(i)^2,
\end{align}
where we use the convention that $\prod_{i=l}^{u}a_i=1$ if $l>u$ for any $a_i\in\R{}$ and $(l,u)\in \Zmbb_{+}^2$.
Then using a similar argument as from Lemma~\ref{lemma:prod.1.minus.beta}, one can establish, for any $i\geq 2$, that
\begin{align*}
    \prod_{j = i}^k (1 - \gamma_j) \leq & ~  \exp\left(- \sum_{j = i}^k \gamma_j \right) = ~ \exp\left(- \frac{\zeta \mu_f^2}{2 L_g^2} \cdot \sum_{j = i}^k \min\left\{ \frac{1}{2}, ~  \frac{c}{j+1} \right\} \right) 
    \\
    = & ~ \exp\left(- \frac{\zeta \mu_f^2}{2 L_g^2} \cdot \frac{\underline{k} - \min\{\underline{k}, i\}}{2} - \frac{\zeta \mu_f^2}{2 L_g^2} \cdot \sum_{j = \max\{\underline{k}, i\}}^k \frac{c}{j + 1}  \right)  \\
    \leq & ~ 
    \exp\left(- \frac{\zeta \mu_f^2}{2 L_g^2} \cdot \frac{\underline{k} - \min\{\underline{k}, i\}}{2} \right) \cdot \left( \frac{\max\{\underline{k}, i\} + 2}{k + 1} \right)^{\zeta \mu_f^2 c / (2 L_g^2)}\\
    = & ~ 
    \exp\left(- \frac{\zeta \mu_f^2}{2 L_g^2} \cdot \frac{\underline{k} - \min\{\underline{k}, i\}}{2} \right) \cdot \left( \frac{\max\{\underline{k}, i\} + 2}{k + 1} \right)^{\theta}.
\end{align*}
Combing the above inequality with \eqref{eq:proxstorm.iter.expand6} we obtain, for any $k \geq \underline{k}$, that 
\begin{align}\label{eq:proxstorm.iter.expand7}
    s_{k + 1}^2 
    \leq & ~ \left( \frac{\underline{k} + 2}{k + 1} \right)^{} \cdot s_{\underline{k}}^2  + c_3 \cdot \sum_{i = \underline{k}}^{k} \left[ \left( \frac{(i + 1) + 2}{k + 1} \right)^{\theta} \right] \frac{c}{i + 1} U(i)^2.
\end{align}
It follows from Theorem~\ref{thm:error.in.gradient} that
\begin{equation}\label{eq:proxstorm.iter.expand8}
    U(i)^2 =
    \begin{cases}
    G^2\left(\frac{\underline{k}+1}{i+2}\right)^{2c}\log\frac{2}{\eta_i} & \text{if $i<\bar{k}$},\\
    G^2\frac{c^2}{i+2}\log\frac{2}{\eta_i} & \text{if $i\geq\bar{k}$,}
    \end{cases}
\end{equation}
where $G=C_2(\sigma + L_g(G_r + G_d)\zeta \ualpha)$ and $\bar{k}=\max\left\{\underline{k},\left\lceil\frac{(\underline{k}+1)^{2c/(2c-1)}}{c^{2/(2c-1)}}-2\right\rceil\right\}$. Then it follows from \eqref{eq:proxstorm.iter.expand8} that
\begin{align}
    & c_3 \cdot \sum_{i = \underline{k}}^{k} \left[ \left( \frac{(i + 1) + 2}{k + 1} \right)^{\theta} \right] \frac{c}{i + 1} U(i)^2 \nonumber\\
    & \leq \frac{c_3\cdot c\cdot G^2}{(k+1)^{\theta}}\left[\sum_{i=\underline{k}}^{\min\{\bar{k}-1, k\}}\frac{(i+3)^{\theta}}{i+1}\frac{(\underline{k}+1)^{2c}}{(i+2)^{2c}}\log\frac{2}{\eta_i} + \sum_{i=\min\{\bar{k}-1, k\}+1}^{k} \frac{(i+3)^{\theta}}{i+1}\frac{c^2}{i+2}\log\frac{2}{\eta_i}\right],\label{eq:proxstorm.iter.expand9}
\end{align}
where we use the convention that $\sum_{i=l}^{u}a_i=0$ if $l>u$ for any $a_i\in\R{}$ and $(l,u)\in \Zmbb_{+}^2$.

It follows from~\eqref{eq:proxstorm.iter.expand9} and $\theta\geq 2$, one obtains, for all $k\geq \underline{k}$, there exists constants $\{C_{30}, C_{31}, C_{32}, C_{34}\}\subset\R{}_{++}$, which are independent of $k$, such that,
\begin{align}
    & c_3 \cdot \sum_{i = \underline{k}}^{k} \left[ \left( \frac{(i + 1) + 2}{k + 1} \right)^{\theta} \right] \frac{c}{i + 1} U(i)^2 \nonumber\\
    & = \frac{c_3\cdot c\cdot G^2}{(k + 1 )^\theta}\left[\sum_{i=\underline{k}}^{\min\{\bar{k}-1, k\}}\frac{(i+3)^\theta}{i+1}\frac{(\underline{k}+1)^{2c}}{(i+2)^{2c}}\log\frac{2}{\eta_i} + \sum_{i=\min\{\bar{k}-1, k\}+1}^{k} \frac{(i+3)^\theta}{i+1}\frac{c^2}{i+2}\log\frac{2}{\eta_i}\right]\nonumber\\
    & =
    \begin{cases}
    \frac{c_3\cdot c\cdot G^2}{(k + 1 )^\theta}\left[\sum_{i=\underline{k}}^{k}\frac{(i+3)^\theta}{i+1}\frac{(\underline{k}+1)^{2c}}{(i+2)^{2c}}\log\frac{2}{\eta_i}\right] & \text{ if $\underline{k}\leq k<\bar k$},\\
    \frac{c_3\cdot c\cdot G^2}{(k + 1 )^\theta}\left[\sum_{i=\underline{k}}^{\bar{k}-1}\frac{(i+3)^\theta}{i+1}\frac{(\underline{k}+1)^{2c}}{(i+2)^{2c}}\log\frac{2}{\eta_i} + \sum_{i=\bar{k}}^{k} \frac{(i+3)^\theta}{i+1}\frac{c^2}{i+2}\log\frac{2}{\eta_i}\right] & \text{ if $\underline{k}\leq \bar k \leq k$,}
    \end{cases}\nonumber\\
    & \leq 
    \begin{cases}
    \frac{c_3\cdot c\cdot G^2(\underline{k}+1)^{2c}\log\frac{2}{\eta_k}}{(k + 1 )^\theta}\left[\sum_{i=\underline{k}}^{k}\frac{(i+3)^\theta}{i+1}\frac{1}{(i+2)^{2c}}\right] & \text{ if $\underline{k}\leq k<\bar k$}, \qquad \text{(due to $\eta_i\geq \eta_k$)}\\
    \frac{c_3\cdot c^3\cdot G^2\log\frac{2}{\eta_k}}{(k + 1 )^\theta}\left[C_{30} + \sum_{i=\bar{k}}^{k} \frac{(i+3)^\theta}{i+1}\frac{1}{i+2}\right] & \text{ if $\underline{k}\leq \bar k \leq k$}, \qquad \text{(due to $\eta_i\geq \eta_k$,  $\bar{k}$, and $\underline{k}$ are both constants)}
    \end{cases}\nonumber\\
    & \leq 
    \begin{cases}
    C_{31}\frac{c_3\cdot c\cdot G^2(\underline{k}+1)^{2c}\log\frac{2}{\eta_k}}{(k + 1 )^\theta}\left(\int_{1}^{k}t^{\theta-1-2c}dt\right) & \text{ if $\underline{k}\leq k<\bar k$}, \\
    C_{32}\frac{c_3\cdot c^3\cdot G^2\log\frac{2}{\eta_k}}{(k + 1 )^\theta}\left(\int_{1}^{k}t^{\theta-2}dt\right)& \text{ if $\underline{k}\leq \bar k \leq k$,}
    \end{cases}\nonumber\\
    & \leq 
    \begin{cases}
    C_{31}\frac{c_3\cdot c\cdot G^2(\underline{k}+1)^{2c}\log\frac{2}{\eta_k}}{(k + 1 )^\theta}\cdot\frac{1}{2c-\theta} & \text{ if $\underline{k}\leq k<\bar k$}, \qquad \text{(due to $c>\theta$)}\\
    C_{32} \frac{c_3\cdot c^3\cdot G^2\log\frac{2}{\eta_k}}{(k + 1)^\theta}k^{\theta-1} \cdot\frac{1}{\theta-1} & \text{ if $\underline{k}\leq \bar k \leq k$,} 
    \end{cases}\nonumber\\
    & \leq 
    \begin{cases}
    C_{31}\frac{c_3\cdot c\cdot G^2(\underline{k}+1)^{2c}\log\frac{2}{\eta_k}}{k + 1 }\cdot\frac{1}{2c-\theta} & \text{ if $\underline{k}\leq k<\bar k$}, \qquad \text{(due to $\theta\geq 2>1$)}\\
    C_{32} \frac{c_3\cdot c^3\cdot G^2\log\frac{2}{\eta_k}}{k + 1} \cdot\frac{1}{\theta-1} & \text{ if $\underline{k}\leq \bar k \leq k$,} 
    \end{cases}\nonumber\\
    &\leq c_3G^2C_{34}\frac{\log\frac{2}{\eta_k}}{k + 1}.\label{eq:proxstorm.iter.expand10}
\end{align}
Combining \eqref{eq:proxstorm.iter.expand7} with \eqref{eq:proxstorm.iter.expand10}, for all $k\geq\underline{k}$, gives
\begin{align}\label{eq:proxstorm.iter.expand11}
    s_{k + 1}^2 
    \leq & ~ \left( \frac{\underline{k} + 2}{k + 1} \right)^{\theta} \cdot s_{\underline{k}}^2  + c_3G^2C_{34}\frac{\log\frac{2}{\eta_k}}{k + 1},
\end{align}
which implies, for all $k\geq\underline{k}$, that
$$
\begin{aligned}
\norm{x_{k}-x^*}^2 
&\leq \left( \frac{\underline{k} + 2}{k} \right)^{\theta} \norm{x_{\underline{k}}-x^*}^2  + c_3G^2C_{34}\frac{2\log\frac{2k}{\eta_0}}{k}\\
&= \left( \frac{\underline{k} + 2}{k} \right)^{\theta}\norm{x_{\underline{k}}-x^*}^2  + \zeta\left(\frac{\mu_f^2}{L_g^4} + \frac{2}{L_g^2}\left(1+\frac{\mu_f}{L_g}\right)^2\right)(\sigma + L_g(G_r + G_d)\zeta \ualpha)^2C_{34}\frac{\log\frac{2k}{\eta_0}}{k}\\
&= {\bar c}_1 \frac{\norm{x_{\underline{k}}-x^*}^2}{k^\theta } + {\bar c}_2\frac{\log\frac{2k}{\eta_0}}{k},
\end{aligned}
$$
where we set $C_3=C_{34}$.
It follows from the definition of $\Ecal^x_k$ and the above result that $\P{\bigcap_{k\geq\underline{k}}^{\infty}\Ecal^x_{k}~\Huge{|}~\bigcap_{k\geq\underline{k}}^{\infty}\Ecal_{k}}=1$.
In conclusion,  for any given $\eta_0 \in (0, 6/\pi^2)$, it follows from the Bayes' Rule and Corollary~\ref{cor:variance.all.iters} that
\begin{align*}
    \P{\bigcap_{k\geq\underline{k}}^{\infty}\Ecal^x_{k}}
    & = \frac{\P{\bigcap_{k\geq\underline{k}}^{\infty}\Ecal^x_{k},\bigcap_{k\geq\underline{k}}^{\infty}\Ecal_{k}}}{\P{\bigcap_{k\geq\underline{k}}^{\infty}\Ecal_{k}~\Huge{|}~\bigcap_{k\geq\underline{k}}^{\infty}\Ecal^x_{k}}}\geq \P{\bigcap_{k\geq\underline{k}}^{\infty}\Ecal^x_{k},\bigcap_{k\geq\underline{k}}^{\infty}\Ecal_{k}} \\
    & = \P{\bigcap_{k\geq\underline{k}}^{\infty}\Ecal^x_{k}~\Huge{|}~\bigcap_{k\geq\underline{k}}^{\infty}\Ecal_{k}}\P{\bigcap_{k\geq\underline{k}}^{\infty}\Ecal_{k}} = \P{\bigcap_{k\geq\underline{k}}^{\infty}\Ecal_{k}}\geq 1 - \eta_0 \pi^2 / 6 > 0,
\end{align*}
which completes the proof.
\end{proof}

\subsection{Proof of Theorem~\ref{thm:id.sufficient.cond}} \label{app:id.sufficient.cond}

\paragraph{Theorem~\ref{thm:id.sufficient.cond}.} 
Given $\alpha>0$, $d\in\R{n}$, and the optimal solution $x^*$ to problem~\eqref{prob:main}, denote $z=x-\alpha d$ and $y=\prox{\alpha r}{z}$. Let  Assumption~\ref{ass:support.id} hold. If
$$
\norm{\frac{[z - x^*]_{g_i}}{\alpha}+\grad_{g_i} f(x^*)}<\delta^* \text{ for all } i\not\in \supp{x^*},
$$
then $\supp{y}\subseteq \supp{x^*}$. Furthermore, if $\norm{y-x^*}< \Delta^*$, then $\supp{x^*}\subseteq \supp{y}$ so that, in fact, $\supp{y} = \supp{x^*}$.
\begin{proof}
We start with the first claim $\supp{y}\subseteq \supp{x^*}$. It follows from Assumption~\ref{ass:support.id} and the triangular inequality that, for all $i\not\in \supp{x^*}$,  one has 
\begin{align*}
\norm{\frac{[z - x^*]_{g_i}}{\alpha}} 
&= \norm{\frac{[z -x^*]_{g_i}}{\alpha}+\grad_{g_i} f(x^*)-\grad_{g_i} f(x^*)}\\
&\leq \norm{\frac{[z -x^*]_{g_i}}{\alpha}+\grad_{g_i} f(x^*)}+\norm{\grad_{g_i} f(x^*)}\\
& < \delta^* + \norm{\grad_{g_i} f(x^*)} \leq \delta_{\min} + \norm{\grad_{g_i} f(x^*)}
< \lambda_i. 
\end{align*}
Since $[x^*]_{g_i}=0$ for all $i\not\in \supp{x^*}$, it follows that $\frac{[z-x^*]_{g_i}}{\alpha}\in \partial r_i([x^*]_{g_i})$\footnote{The subdifferential is given by $\partial \norm{x}=\{v\in\R{n}~|~\norm{v}\leq 1\}$.}. It follows from the optimality condition for the proximal problem~\citep[Theorem 6.39]{beck2017first} that this is true if and only if 
$[x^*]_{g_i} = \prox{\alpha r_i}{[z]_{g_i}}$ for all $i\not\in \supp{x^*}$,
which further implies $[y]_{g_i}=[x^*]_{g_i}=0$ for all $i\not\in \supp{x^*}$. Consequently, $(\supp{x^*})^c \subseteq (\supp{y})^c$, which implies $\supp{y}\subseteq \supp{x^*}$. 

Now we prove the second claim $\supp{x^*} \subseteq \supp{y}$. Note that $\norm{[y-x^*]_{g_i}} \leq \norm{y_k-x^*}$ for any $i\in[\ngrp]$. Therefore, when $\norm{y-x^*}<\Delta^*$, for $i\in \supp{x^*}$,
$[y]_{g_i}$ cannot be 0 for all $i\in \supp{x^*}$. Otherwise, $\Delta^*\leq \norm{[x^*]_{g_i}}<\Delta^*$ for $i\in \supp{x^*}$. This proves that $\supp{x^*}\subseteq \supp{y_k}$.
\end{proof}

\subsection{Proof of Theorem~\ref{thm:support.identification}} \label{app:support.identification}

\paragraph{Theorem~\ref{thm:support.identification}.} 
Let Assumption~\ref{ass:generic.assumptions}--Assumption~\ref{ass:support.id} hold, $\zeta \in (0, 2)$, $\theta\geq 2$, $c = (2\theta L_g^2) / (\zeta \mu_f^2) > 2$,  and $\underline{k}=\lceil 2c-1\rceil$. Consider the sequence $\{y_k\}$ of Algorithm~\ref{algo:main} and define the event $\Ecal_{k}^{\text{id}}=\{\Scal(y_k)=\Scal(x^*)\}$ for all $k\geq 1$. Then, there exists constants $\{C_{41}, C_{42}\}\subseteq\R{n}_{++}$ that are independent of $k$,
$k_{\delta^*}=(C_{41}/\delta^*)^{4}$ and $k_{\Delta^*}=(C_{42}/\Delta^*)^{4}$ such that, with $K := \max\{k_{\delta^*}, k_{\Delta^*}, \underline{k}$\}, it follows that
$$
\P{\bigcap_{k\geq K}^{\infty} \Ecal_{k}^{\text{id}} }\geq 1-\frac{\eta_0\pi^2}{6} > 0.
$$
\begin{proof}
Denote $z_k = x_k-\alpha_kd_k$ for all $k\geq 1$, then it follows from Assumption~\ref{ass:generic.assumptions}(\ref{ass:generic.assumptions.2}) and the triangular inequality that
\begin{align}
\norm{\frac{z_k-x^*}{\alpha_k}+\grad f(x^*)}
& \leq \frac{1}{\alpha_k}\norm{x_k-x^*} + \norm{d_k -\grad f(x^*) } \nonumber\\
& \leq \frac{1}{\alpha_k}\norm{x_k-x^*} + \norm{d_k -\grad f(x_k) } + \norm{\grad f(x_k) -\grad f(x^*) }\nonumber\\
& \leq \left(\frac{1}{\alpha_k}+L_g\right)\norm{x_k-x^*} + \norm{d_k -\grad f(x_k) }.\label{eq:finite.id.1}
\end{align}
\textbf{Conditioning on the events $\bigcap_{k\geq \underline{k}}^{\infty} \Ecal_i$ and $\bigcap_{k\geq \underline{k}}^{\infty} \Ecal^x_i$ happening} (with $\Ecal_k$ defined in Theorem~\ref{thm:error.in.gradient} and $\Ecal^x_k$ defined in Theorem~\ref{thm:iterate.complexity}), it follows from $\alpha_k\equiv \ualpha$ for all $k$ (Assumption~\ref{ass:algo.choice}), Corollary~\ref{cor:variance.all.iters}, and  Theorem~\ref{thm:iterate.complexity} that, there exists a constant $C_{41}>0$ that is independent of $k$, for all $k\geq\underline{k}$,
\begin{equation} \label{eq:id.support.ub}
    \left(\frac{1}{\alpha_k}+L_g\right)\norm{x_k-x^*} + \norm{d_k -\grad f(x_k) } \leq C_{41}\sqrt{\frac{\log k}{k}}.
\end{equation}
Combining \eqref{eq:finite.id.1} and \eqref{eq:id.support.ub}, we know for all $k\geq k_{\delta^*}$ that $C_{41}\sqrt{\frac{\log k}{k}}\leq C_{41}/k^4<C_{41}/k_{\delta^*}^4=\delta^*$
\footnote{We use the inequality $\sqrt{\frac{\log x}{x}}<\frac{1}{x^{1/4}}$ for all $x>1$.}. 
Together with Theorem~\ref{thm:id.sufficient.cond} and the definition of $y_k$ (line~\ref{line:yk} of Algorithm~\ref{algo:main}), we have $\supp{y_k}\subseteq\supp{x^*}$ for all $k\geq\max\{k_{\delta^*},\underline{k}\}$. 

It follows from the  non-expansiveness~\citep[Theorem 6.42]{beck2017first} of the proximal operator, $x^*=\prox{\alpha_kr}{x^*-\alpha_k\grad f(x^*)}$, the definition of $y_k$ (line~\ref{line:yk} of Algorithm~\ref{algo:main}), and the triangular inequality that
\begin{align}
\norm{y_k-x^*}
& = \norm{\prox{\alpha_kr}{x_k-\alpha_k\grad f(x_k)}-\prox{\alpha_kr}{x^*-\alpha_k\grad f(x^*)}}\nonumber\\
& \leq \norm{(x_{k}-x^*)-\alpha_k(d_k-\grad f(x^*))}\nonumber\\
& \leq \norm{x_{k}-x^*} + \alpha_k\norm{d_k-\grad f(x^*)}\label{eq:finite.id.2}.
\end{align}
Again, \textbf{conditioning on the events $\bigcap_{k\geq \underline{k}}^{\infty} \Ecal_i$ and $\bigcap_{k\geq \underline{k}}^{\infty} \Ecal^x_i$ happening}, it follows from $\alpha_k\equiv \ualpha$ for all $k$ (Assumption~\ref{ass:algo.choice}), Corollary~\ref{cor:variance.all.iters}, Theorem~\ref{thm:iterate.complexity}, and \eqref{eq:finite.id.2} that, there exist a constant $C_{42}>0$ that is independent of $k$, such that for all $k\geq\underline{k}$, $\norm{y_k-x^*}\leq C_{42}\sqrt{\frac{\log k}{k}}$. Therefore, when $k\geq k_{\Delta^*}$, it follows that $C_{42}\sqrt{\frac{\log k}{k}} \leq C_{42}/k^4<C_{42}/k_{\Delta^*}^4 = \Delta^*$.
Together with Theorem~\ref{thm:id.sufficient.cond} and the definition of $y_k$ (line~\ref{line:yk} of Algorithm~\ref{algo:main}), we have $\supp{x^*}\subseteq\supp{y_k}$ for all $k\geq\max\{k_{\Delta^*}, \underline{k}\}$.
Therefore, when $k\geq K=\max\{k_{\delta^*}, k_{\Delta^*}, \underline{k}\}$, together with the fact that $\P{\bigcap_{k\geq K}^{\infty} \Ecal_{k}^{\text{id}}~\Huge{|}~ \bigcap_{k\geq \underline{k}}^{\infty}\Ecal_i, \bigcap_{k\geq \underline{k}}^{\infty} \Ecal^x_i}=1$, it follows that 
\begin{align*}
    \P{\bigcap_{k\geq K}^{\infty} \Ecal_{k}^{\text{id}}} 
    &= \frac{\P{\bigcap_{k\geq K}^{\infty} \Ecal_{k}^{\text{id}}, \bigcap_{k\geq \underline{k}}^{\infty}\Ecal_i, \bigcap_{k\geq \underline{k}}^{\infty} \Ecal^x_i}}{\P{\bigcap_{k\geq \underline{k}}^{\infty}\Ecal_i, \bigcap_{k\geq \underline{k}}^{\infty} \Ecal^x_i ~\Huge{|}~ \bigcap_{k\geq K}^{\infty} \Ecal_{k}^{\text{id}} }}\geq \P{\bigcap_{k\geq K}^{\infty} \Ecal_{k}^{\text{id}}, \bigcap_{k\geq \underline{k}}^{\infty}\Ecal_i, \bigcap_{k\geq \underline{k}}^{\infty} \Ecal^x_i} \\
    &= \P{\bigcap_{k\geq K}^{\infty} \Ecal_{k}^{\text{id}}~\Huge{|}~ \bigcap_{k\geq \underline{k}}^{\infty}\Ecal_i, \bigcap_{k\geq \underline{k}}^{\infty} \Ecal^x_i}\P{\bigcap_{k\geq \underline{k}}^{\infty}\Ecal_i, \bigcap_{k\geq \underline{k}}^{\infty} \Ecal^x_i}\\
    &=\P{\bigcap_{k\geq \underline{k}}^{\infty}\Ecal_i, \bigcap_{k\geq \underline{k}}^{\infty} \Ecal^x_i}\geq 1 - \frac{\eta_0 \pi^2}{6},
\end{align*}
which completes the proof.
\end{proof}

\subsection{Proofs for additional lemmas}

\begin{lemma}\label{lemma:ass.implication}
Denote $\Fcal_1 = \sigma(x_1)$ and, for all $k \geq 2$, denote $\Fcal_k$ as the $\sigma$-algebra generated by the random variables $\{\{\Xi_{1,i}\}_{i=1}^m,\dots,\{\Xi_{(k-1),i}\}_{i=1}^m\}$ (of which $\{\{\xi_{1,i}\}_{i=1}^{m},\dots,\{\xi_{(k-1),i}\}_{i=1}^{m}\}$ is a realization) so that $\{\Fcal_k\}$ is a filtration.
If (i) there exists a constant $c_e>0$ such that for all $k\geq 1$, $\Pmbb\{\norm{d_k-\grad f(x_k)}\leq c_e~|~{\Fcal_k}\} = 1$ and  (ii) there exists a constant $c_{\alpha}$ such that for a given $\alpha>0$ and all $k\geq 1$, $\Pmbb\{\chi(x_k;\alpha)\leq c_{\alpha}~|~{\Fcal_k}\}=1$, then there exists a constant $G_d>0$ such that for all $k\geq 1$, it holds that $\Pmbb\{{\norm{d_k}\leq G_d}~|~{\Fcal_k}\}=1$.
\end{lemma}
\begin{proof}
To see why the implication holds, we define $\tilde y_k=\prox{\alpha r}{x_k-\alpha \grad f(x_k)}$ and we make an algorithmic choice $\alpha_k\equiv \alpha$ for all $k$. Since $y_{k}=\prox{\alpha r}{x_{k}-\alpha d_{k}}$, then $\frac{x_{k}-y_{k}}{\alpha}-d_{k}\in\partial r(y_{k})$. It follows from  Assumption~\ref{ass:proxstorm.varinace.reduced}.\ref{ass:proxstorm.varinace.reduced.3} that  $\norm{\frac{x_{k}-y_{k}}{\alpha}-d_{k}}\leq G_r$. By the triangle inequality, we have 
$$
\begin{aligned}
\norm{d_k}
& \leq G_r + \frac{\norm{x_k-y_k}}{\alpha}\leq G_r + \frac{\norm{x_k-\tilde y_k}}{\alpha} + \frac{\norm{y_k-\tilde y_k}}{\alpha}\\
& = G_r+\chi_k(\alpha) + \frac{\norm{\prox{\alpha r}{x_{k}-\alpha d_{k}}-\prox{\alpha r}{x_{k}-\alpha \grad f(x_k)}}}{\alpha}\\
& \leq G_r+\chi_k(\alpha) + \norm{d_k-\grad f(x_k)}\leq G_r + c_\alpha + c_e,
\end{aligned}
$$
where the penultimate inequality holds by the non-expansiveness of the proximal operator\citep[Theorem 6.42]{beck2017first}.
\end{proof}

\begin{lemma}\label{lemma:rda.related}
Consider the \rda{} algorithm 
with its update defined as
$$
x_{k+1} = \arg\min_{x\in\R{n}}\left\{d_k^Tx + r(x)+ \frac{\rho_k}{k}\norm{x}^2\right\} \text{ with } d_k = \frac{k-1}{k}d_{k-1} + \frac{1}{k}\grad\ell(x_k;\xi_k),
$$
where $\xi_k$ is a i.i.d sample from $\Pcal$.
\begin{enumerate}
    \item If $\rho_k=\ualpha\sqrt{k}$ for a given $k\geq 1$ and $\ualpha$ is defined as in Assumption~\ref{ass:algo.choice}, then the update can be equivalently written as
    $$
    x_{k+1}=\prox{\alpha_kr}{-\alpha_kd_k} \text{ with } \alpha_k = \frac{\sqrt{k}}{\underline{\alpha}}.
    $$\label{lemma:rda.related.1}
    \item Assume $r$ is $\mu_r>0$ strongly convex. Further, assume that there are constants $\{G,D\}\subset (0,\infty)$ such that, for all $k\geq 1$, it holds that $\norm{\grad\ell(x_k;\xi_k)}\leq G$ and $\norm{x^*}\leq D$.
    If $\rho_k=\ualpha\sqrt{k}$, then
    $$
    \E{}{\norm{x_k-x^*}^2}\leq \frac{2\left(\ualpha D^2+G^2/\ualpha\right)}{\mu_r}\frac{1}{ \sqrt{k}}.
    $$ Moreover, for any $\epsilon>0$ and $k\geq 1$, it holds that
    $$
    \P{\norm{x_k-x^*}\geq \epsilon} \leq\sqrt{ \frac{2\left(\ualpha D^2+G^2/\ualpha\right)}{\mu_r\epsilon^2}}\frac{1}{k^{1/4}}.
    $$\label{lemma:rda.related.2}
    \item  Assume $f$ and $r$ are $\mu_f > 0$ and $\mu_r>0$ strongly convex, respectively. Further, assume that there are constants $\{G,D\}\subset (0,\infty)$ such that, for all $k\geq 1$, it holds that $\norm{\grad\ell(x_k;\xi_k)}\leq G$ and $\norm{x^*}\leq D$.
    If $\rho_k=\ualpha\sqrt{k}$, then for all $k\geq 1$, it holds that
    $$
    \P{\supp{x_{k+1}}=\supp{x^*}} \geq 1 -\eta_k^{\rda{}},
    $$
    where $\eta_k^{\rda{}}=\max\left\{\Ocal\left(\frac{1}{\delta^*\cdot k^{1/4}}\right), \Ocal\left(\frac{1}{\Delta^*\cdot k^{1/4}}\right)\right\}$.\label{lemma:rda.related.3}
\end{enumerate}
\end{lemma}

\begin{proof}
For part 1, let $\beta_k=\ualpha\sqrt{k}$. Then
\begin{align*}
    \arg\min_{x\in\R{n}}\left\{d_k^Tx + r(x)+ \frac{\rho_k}{k}\norm{x}^2\right\} & = \arg\min_{x\in\R{n}}\left\{d_k^Tx + r(x)+ \frac{1}{\frac{1}{\ualpha}\sqrt{k}}\norm{x}^2\right\} \\
    & = \arg\min_{x\in\R{n}}\left\{d_k^Tx + r(x)+ \frac{1}{\alpha_k}\norm{x}^2\right\} \\
    & = \arg\min_{x\in\R{n}}\left\{\frac{1}{\alpha_k}\norm{x+\alpha_kd_k}^2 + r(x)\right\} = \prox{\alpha_kr}{-\alpha_kd_k}.
\end{align*}
For part 2, it follows from \citet[Equation 22]{xiao2009dual} and \citet[Corollary 2]{xiao2009dual} that for all $k\geq 1$,
$$
\E{}{\norm{x_k-x^*}^2}\leq \frac{2}{\mu_r k}\left(\ualpha D^2+G^2/\ualpha\right)\sqrt{k} = \frac{2\left(\ualpha D^2+G^2/\ualpha\right)}{\mu_r}\frac{1}{ \sqrt{k}}.
$$
It follows from Jensen' inequality that
$$
\E{}{\norm{x_k-x^*}}\leq \sqrt{\E{}{\norm{x_k-x^*}^2}} \leq \sqrt{\frac{2\left(\ualpha D^2+G^2/\ualpha\right)}{\mu_r}}\frac{1}{k^{1/4}},
$$
which together with the Markov inequality implies that
$$
    \P{\norm{x_k-x^*}\geq \epsilon} \leq\sqrt{ \frac{2\left(\ualpha D^2+G^2/\ualpha\right)}{\mu_r\epsilon^2}}\frac{1}{k^{1/4}}.
$$
For part 3,
consider three events  $\Ecal^{\rda{}}_{k,1} := \{\norm{d_k-\grad f(x^*)}< \delta^* / 2 $\}, $\Ecal^{\rda{}}_{k,2} := \{ (1/\alpha_k + L_g)\norm{x_k-x^*}< \delta^* / 2\}$, and $\Ecal^{\rda{}}_{k,3} := \{ \norm{x_{k+1}-x^*}< \Delta^*\}$.  
It follows from~\citet[Theorem 11, equation (31)]{lee2012manifold} and part 2 of this lemma that
\begin{align}
    & \P{\left(\Ecal^{\rda{}}_{k,1}\right)^c}\leq \Ocal\left(\frac{1}{\delta^*\cdot k^{1/4}}\right), \label{eq:rda.1}\\
    & \P{\left(\Ecal^{\rda{}}_{k,2}\right)^c}\leq \Ocal\left(\frac{1}{\delta^*\cdot k^{3/4}}\right), \ \  \text{and} \label{eq:rda.2} \\
    & \P{\left(\Ecal^{\rda{}}_{k,2}\right)^c}\leq \Ocal\left(\frac{1}{\Delta^*\cdot k^{1/4}}\right).\label{eq:rda.3}
\end{align}
It follows from the union bound and \eqref{eq:rda.1}--\eqref{eq:rda.3} that
\begin{align*}
    \P{\Ecal^{\rda{}}_{k,1} \bigcap \Ecal^{\rda{}}_{k,2} \bigcap \Ecal^{\rda{}}_{k,3} }
    & = 1 - \P{\left(\Ecal^{\rda{}}_{k,1}\right)^c \bigcup \left(\Ecal^{\rda{}}_{k,2}\right)^c \bigcup \left(\Ecal^{\rda{}}_{k,3}\right)^c}\\
    &\geq 1 - \left(\P{\left(\Ecal^{\rda{}}_{k,1}\right)^c} + \P{\left(\Ecal^{\rda{}}_{k,2}\right)^c} + \P{\left(\Ecal^{\rda{}}_{k,3}\right)^c}\right) \\
    & = 1 - \max\left\{\Ocal\left(\frac{1}{\delta^*\cdot k^{1/4}}\right), \Ocal\left(\frac{1}{\Delta^*\cdot k^{1/4}}\right)\right\},
\end{align*}
which together with Theorem~\ref{thm:id.sufficient.cond} implies that, for any chosen $k\geq 1$,
\begin{equation*}
\P{\supp{x_{k+1}}=\supp{x^*}}\geq 1 - \eta_k^{\rda{}}, 
\end{equation*}
which completes the proof.
\end{proof}


\section{Experiments}~\label{appendix:exp}

\subsection{Discussions on the performance gaps in different methods}\label{appendix:exp.performance}

First, \proxsvrg{} performs poorly on test instances induced by the datasets phishing, rcv1, real-sim, and news20. It can be checked that these datasets cover different sample sizes and decision variable dimensions. We attribute the cause of poor performance to the inner loop length parameter, which is difficult to choose to work on all test instances. In the experiments, we set it to 1 for all cases to follow the original paper's experimental setting~\citep{xiao2014proximal}.

Second, \saga{} performed quite well on the first 32 test instances, where the memory limit is not violated, and failed on the remaining 48 test instances (marked as the darkest red) because the program terminates immediately due to memory limits being exceeded.

Third, \rda{} appears to perform poorly compared with \pstorm{} (\spstorm{}) because the prox step of \rda{} only applies to its initial point $x_0=0$ with updated search direction $-\alpha_kd_k$ (see Lemma~\ref{lemma:rda.related}(1)), whereas \pstorm{} (\spstorm{})  applies the prox step at the up-to-date iterate $x_k$. 

Finally, one can see that \spstorm{} significantly outperforms \pstorm{}. We attribute this to a combination of the stabilization that we introduced and that the step size for \pstorm{} was designed for nonconvex problems (for our tests, nonetheless, we fine-tuned the step size for \pstorm{} to be fair).

\subsection{Additional results}\label{appendix:exp.additional}

We visualize three metrics: the distance to the optimal solution $\norm{x_k-x^*}$ ($x^*$ is obtained by the \farsagroup{} algorithm), the error in the gradient evaluation $\epsilon_k$ (defined in Theorem~\ref{thm:error.in.gradient}), and the sparse structure of major iterates, 
which can be found in the first, second, and third column of \autoref{fig:convergence}, respectively.  The first metric measures the convergence speed of $\{x_k\}$, the second metric shows how fast the error in the stochastic gradient estimator $d_k$ (defined in Algorithm~\ref{algo:main} line~\ref{line:storm}) diminishing to zero, and the third metric visualizes the progress made with respect to support identification.

For demonstration, we only show results on six moderate-size datasets with randomly picked problem parameters $\Lambda=0.1$ and number of groups $\lfloor 0.5n \rfloor$. We remark that in some plots, lines that represent different algorithms could visually overlap. For example, the green line (\spstorm{}) and purple line (\saga{}) overlap in the first column for dataset phishing and rcv1.\footnote{The reason is that the numerical difference between $\norm{x_k-x^*}$ for \spstorm{} and \saga{} is of order $10^{-2}$.} We also emphasize that \saga{} does not appear in the \autoref{fig:convergence}(f) due to memory limitation.

From the first and the second column of \autoref{fig:convergence}, it can be observed that the rates at which the $x_k$ converges to $x^*$ and $\epsilon_k$ converges to $0$ seem to be bounded by $\Ocal(\sqrt{\log k/k})$, which matches our theoretical results in Theorem~\ref{thm:error.in.gradient} and Theorem~\ref{thm:iterate.complexity}. We can also observe that for the relatively large datasets rcv1 and real-sim, 1000 data passes is not enough to obtain accurate estimates of $x^*$, but a decent ratio of zeros groups is identified nonetheless.

\begin{figure*}
     \caption{Visualization of three metrics: the distance to optimal solution $\norm{x_k-x^*}$ (the first column), error in the gradient evaluation $\epsilon_k$ (the second column), and the progress of support identification on different datasets (the third column). We added a dotted reference line corresponding to $\sqrt{\log k / k}$ (for $k\geq 2$) for the plots in the first and second columns. In addition, we added a horizontal black reference line for the plots in the third column to indicate the number of zero groups at the optimal solution $x^*$. } 
     \resizebox{0.83\textwidth}{!}{
     \begin{subfigure}[b]{\textwidth}
         \includegraphics[height=3.5cm]{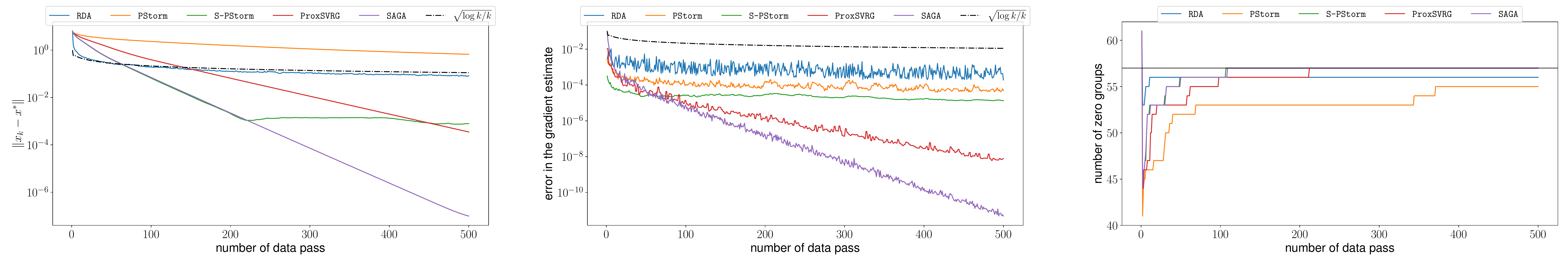}
         \caption{a9a}
     \end{subfigure}
    }
    \resizebox{0.83\textwidth}{!}{
     \begin{subfigure}[b]{\textwidth}
         \includegraphics[height=3.5cm]{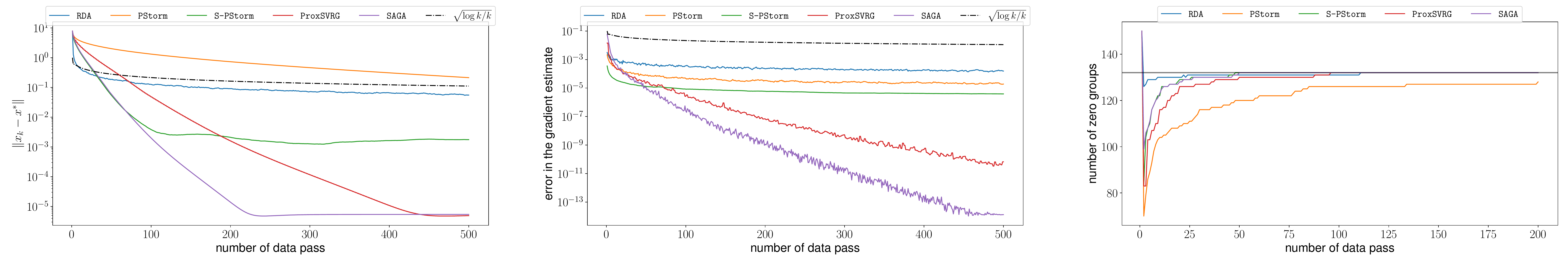}
         \caption{w8a}
     \end{subfigure}
    }
    \resizebox{0.83\textwidth}{!}{
     \begin{subfigure}[b]{\textwidth}
         \includegraphics[height=3.5cm]{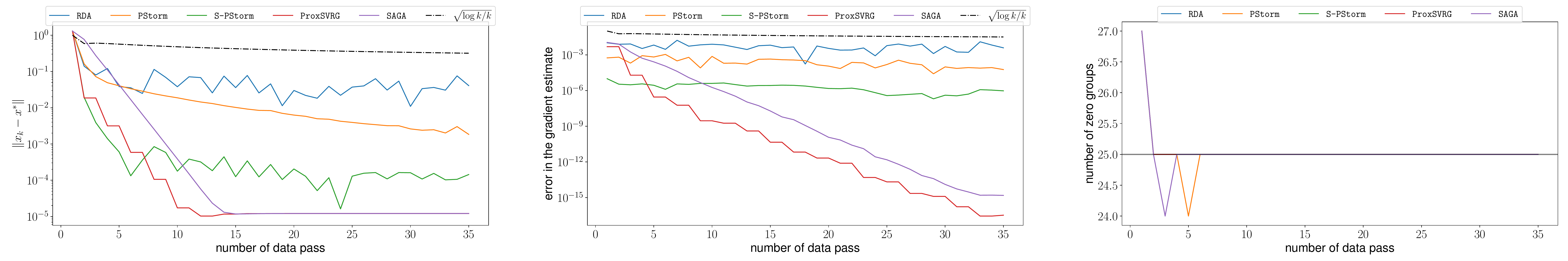}
         \caption{covtype}
     \end{subfigure}
    }    
    \resizebox{0.83\textwidth}{!}{
     \begin{subfigure}[b]{\textwidth}
         \includegraphics[height=3.5cm]{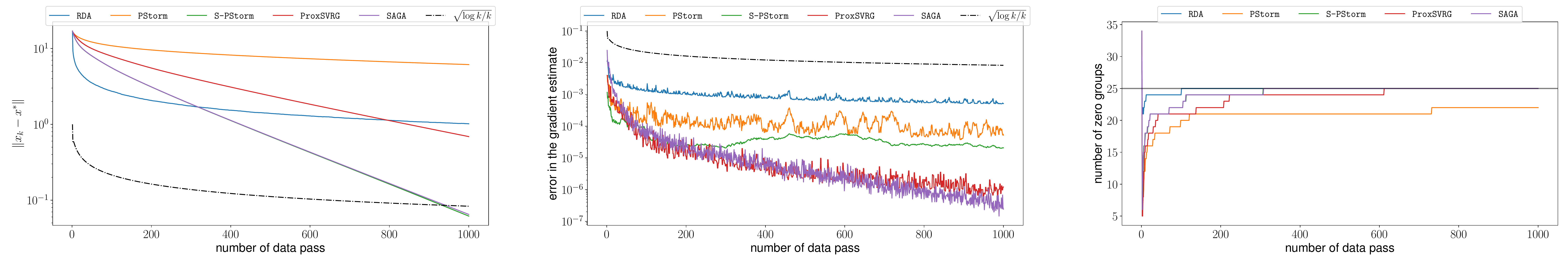}
         \caption{phishing}
     \end{subfigure}
    }
     \resizebox{0.83\textwidth}{!}{
     \begin{subfigure}[b]{\textwidth}
         \includegraphics[height=3.5cm]{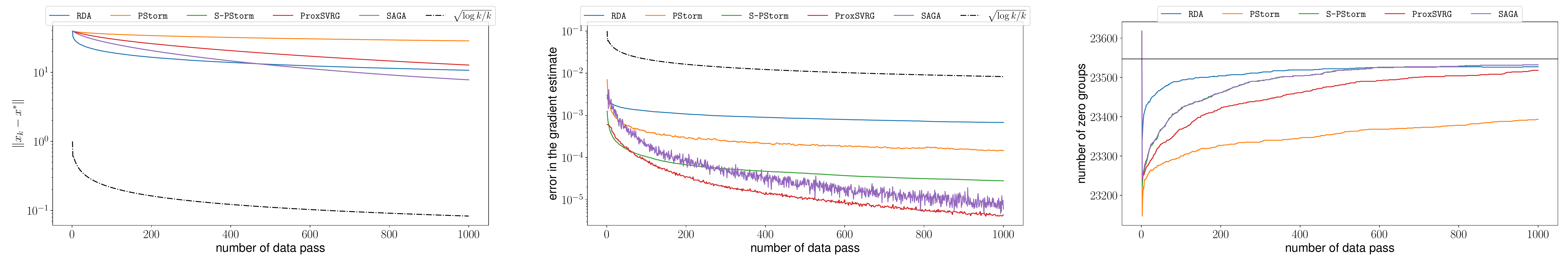}
         \caption{rcv1}
     \end{subfigure}
    }    
    \resizebox{0.83\textwidth}{!}{
     \begin{subfigure}[b]{\textwidth}
         \includegraphics[height=3.5cm]{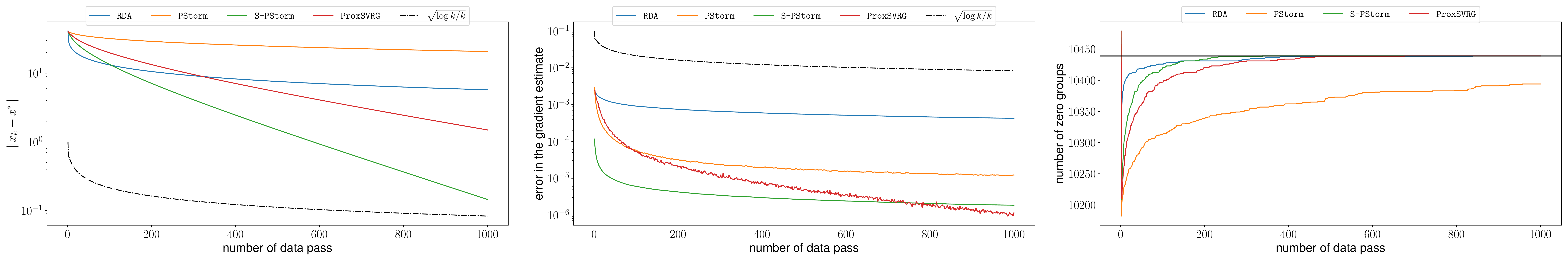}
         \caption{real-sim}
     \end{subfigure}
    }    
    \label{fig:convergence}
\end{figure*}

\newpage
Finally, we provide the raw data for the metrics of total identification (\autoref{tab:total_id}), first identification (\autoref{tab:first_id}), first consistent identification (\autoref{tab:first_cst_id}), and the last iterate support recovery (\autoref{tab:last_nz}); for an explanation of their precise meaning, revisit Section~\ref{sec:num-results-metrics}. All results (excluding \farsagroup{} which is a deterministic algorithm) are reported as the average of three independent runs. In all tables, the problem instance is formatted as (dataset name)-(value of $\Lambda$)-(ratio of \# of groups).

We remark that NaN represents that a particular method failed to identify the support within 1000 data pass. We also removed the instances that all five methods failed to identify the support. 

\begin{table}[!h]
\centering
\caption{Total number of support identifications.}
\label{tab:total_id}
\scalebox{0.8}{
\begin{tabular}{lrrrrr}
\toprule
              instance & \proxsvrg{} & \saga{} & \rda{} & \pstorm{} & \spstorm \\
\midrule
          a9a-0.1-0.25 &       976.0 &   987.0 & 1000.0 &     899.0 &    988.0 \\
           a9a-0.1-0.5 &       790.0 &   893.0 &   75.0 &       NaN &    895.0 \\
          a9a-0.1-0.75 &       932.0 &   965.0 &  998.0 &     547.0 &    966.0 \\
           a9a-0.1-1.0 &       932.0 &   965.0 &  998.0 &     547.0 &    966.0 \\
         a9a-0.01-0.25 &         NaN &     NaN &    NaN &       NaN &     93.0 \\
          a9a-0.01-0.5 &       578.0 &   788.0 &  706.0 &       NaN &    786.0 \\
         a9a-0.01-0.75 &       652.0 &   825.0 &   84.0 &       NaN &    825.0 \\
          a9a-0.01-1.0 &       652.0 &   825.0 &   84.0 &       NaN &    825.0 \\
      covtype-0.1-0.25 &       998.0 &   999.0 & 1000.0 &     998.0 &   1000.0 \\
       covtype-0.1-0.5 &      1000.0 &   999.0 & 1000.0 &     999.0 &   1000.0 \\
      covtype-0.1-0.75 &      1000.0 &   999.0 & 1000.0 &     999.0 &   1000.0 \\
       covtype-0.1-1.0 &      1000.0 &   999.0 & 1000.0 &     999.0 &   1000.0 \\
     covtype-0.01-0.25 &      1000.0 &  1000.0 & 1000.0 &    1000.0 &   1000.0 \\
      covtype-0.01-0.5 &      1000.0 &  1000.0 & 1000.0 &    1000.0 &   1000.0 \\
     covtype-0.01-0.75 &       994.0 &   991.0 &  999.0 &     837.0 &   1000.0 \\
      covtype-0.01-1.0 &       994.0 &   991.0 &  999.0 &     837.0 &   1000.0 \\
     phishing-0.1-0.25 &       792.0 &   896.0 &  991.0 &       NaN &    895.0 \\
      phishing-0.1-0.5 &       390.0 &   694.0 &  901.0 &       NaN &    695.0 \\
     phishing-0.1-0.75 &       420.0 &   710.0 &  502.0 &       NaN &    710.0 \\
      phishing-0.1-1.0 &       326.0 &   662.0 &   28.0 &       NaN &    667.0 \\
          w8a-0.1-0.25 &       960.0 &   979.0 &  997.0 &     749.0 &    980.0 \\
           w8a-0.1-0.5 &       906.0 &   952.0 &  891.0 &     120.0 &    954.0 \\
          w8a-0.1-0.75 &       886.0 &   942.0 &  951.0 &       NaN &    942.0 \\
           w8a-0.1-1.0 &       886.0 &   942.0 &  951.0 &       NaN &    942.0 \\
          w8a-0.01-0.5 &       164.0 &   581.0 &    NaN &       NaN &    580.0 \\
         w8a-0.01-0.75 &         NaN &   185.0 &    NaN &       NaN &    195.0 \\
          w8a-0.01-1.0 &         NaN &   185.0 &    NaN &       NaN &    195.0 \\
     real-sim-0.1-0.25 &         NaN &     NaN &    NaN &       NaN &    212.0 \\
      real-sim-0.1-0.5 &       326.0 &     NaN &  163.0 &       NaN &    664.0 \\
       news20-0.1-0.25 &        38.0 &     NaN &    NaN &       NaN &     26.0 \\
        news20-0.1-0.5 &         6.0 &     NaN &    NaN &       NaN &    312.0 \\
 url-combined-0.1-0.25 &         4.0 &     NaN &    NaN &       NaN &      2.0 \\
 avazu-app.tr-0.1-0.25 &         6.0 &     NaN &    4.0 &       3.0 &      3.0 \\
  avazu-app.tr-0.1-0.5 &         2.0 &     NaN &    3.0 &       2.0 &      2.0 \\
 avazu-app.tr-0.1-0.75 &         2.0 &     NaN &    2.0 &       2.0 &      2.0 \\
  avazu-app.tr-0.1-1.0 &         NaN &     NaN &    2.0 &       1.0 &      1.0 \\
avazu-app.tr-0.01-0.25 &         6.0 &     NaN &    3.0 &       NaN &      3.0 \\
 avazu-app.tr-0.01-0.5 &         2.0 &     NaN &    2.0 &       NaN &      2.0 \\
avazu-app.tr-0.01-0.75 &         2.0 &     NaN &    1.0 &       NaN &      2.0 \\
 avazu-app.tr-0.01-1.0 &         NaN &     NaN &    2.0 &       NaN &      1.0 \\
\bottomrule
\end{tabular}
}
\end{table}

\begin{table}
\centering
\caption{First support identification.}
\label{tab:first_id}
\scalebox{0.8}{
\begin{tabular}{lrrrrr}
\toprule
         instance & \proxsvrg{} & \saga{} & \rda{} & \pstorm{} & \spstorm{} \\
\midrule
          a9a-0.1-0.25 &     25.0 &     14.0 &   1.0 &  102.0 &     13.0 \\
           a9a-0.1-0.5 &    211.0 &    108.0 & 660.0 &    NaN &    106.0 \\
          a9a-0.1-0.75 &     69.0 &     36.0 &   3.0 &  454.0 &     35.0 \\
           a9a-0.1-1.0 &     69.0 &     36.0 &   3.0 &  454.0 &     35.0 \\
         a9a-0.01-0.25 &      NaN &      NaN &   NaN &    NaN &    908.0 \\
          a9a-0.01-0.5 &    423.0 &    213.0 & 264.0 &    NaN &    215.0 \\
         a9a-0.01-0.75 &    349.0 &    176.0 &  77.0 &    NaN &    176.0 \\
          a9a-0.01-1.0 &    349.0 &    176.0 &  77.0 &    NaN &    176.0 \\
      covtype-0.1-0.25 &      3.0 &      1.0 &   1.0 &    1.0 &      1.0 \\
       covtype-0.1-0.5 &      1.0 &      1.0 &   1.0 &    1.0 &      1.0 \\
      covtype-0.1-0.75 &      1.0 &      1.0 &   1.0 &    1.0 &      1.0 \\
       covtype-0.1-1.0 &      1.0 &      1.0 &   1.0 &    1.0 &      1.0 \\
     covtype-0.01-0.25 &      1.0 &      1.0 &   1.0 &    1.0 &      1.0 \\
      covtype-0.01-0.5 &      1.0 &      1.0 &   1.0 &    1.0 &      1.0 \\
     covtype-0.01-0.75 &      3.0 &      5.0 &   1.0 &    1.0 &      1.0 \\
      covtype-0.01-1.0 &      3.0 &      5.0 &   1.0 &    1.0 &      1.0 \\
     phishing-0.1-0.25 &    209.0 &    105.0 &  10.0 &    NaN &    106.0 \\
      phishing-0.1-0.5 &    611.0 &    307.0 & 100.0 &    NaN &    306.0 \\
     phishing-0.1-0.75 &    581.0 &    291.0 & 492.0 &    NaN &    291.0 \\
      phishing-0.1-1.0 &    675.0 &    339.0 & 734.0 &    NaN &    334.0 \\
          w8a-0.1-0.25 &     41.0 &     22.0 &   4.0 &  252.0 &     21.0 \\
           w8a-0.1-0.5 &     95.0 &     49.0 & 110.0 &  881.0 &     47.0 \\
          w8a-0.1-0.75 &    115.0 &     59.0 &  38.0 &    NaN &     59.0 \\
           w8a-0.1-1.0 &    115.0 &     59.0 &  38.0 &    NaN &     59.0 \\
          w8a-0.01-0.5 &    837.0 &    420.0 &   NaN &    NaN &    421.0 \\
         w8a-0.01-0.75 &      NaN &    816.0 &   NaN &    NaN &    806.0 \\
          w8a-0.01-1.0 &      NaN &    816.0 &   NaN &    NaN &    806.0 \\
     real-sim-0.1-0.25 &      NaN &      NaN &   NaN &    NaN &    789.0 \\
      real-sim-0.1-0.5 &    675.0 &      NaN & 838.0 &    NaN &    337.0 \\
       news20-0.1-0.25 &    963.0 &      NaN &   NaN &    NaN &    504.0 \\
        news20-0.1-0.5 &    995.0 &      NaN &   NaN &    NaN &    523.0 \\
 url-combined-0.1-0.25 &      9.0 &      NaN &   NaN &    NaN &      5.0 \\
 avazu-app.tr-0.1-0.25 &      3.0 &      NaN &   1.0 &    1.0 &      1.0 \\
  avazu-app.tr-0.1-0.5 &      3.0 &      NaN &   1.0 &    1.0 &      1.0 \\
 avazu-app.tr-0.1-0.75 &      3.0 &      NaN &   1.0 &    1.0 &      1.0 \\
  avazu-app.tr-0.1-1.0 &      NaN &      NaN &   1.0 &    1.0 &      1.0 \\
avazu-app.tr-0.01-0.25 &      3.0 &      NaN &   1.0 &    NaN &      1.0 \\
 avazu-app.tr-0.01-0.5 &      3.0 &      NaN &   1.0 &    NaN &      1.0 \\
avazu-app.tr-0.01-0.75 &      3.0 &      NaN &   2.0 &    NaN &      1.0 \\
 avazu-app.tr-0.01-1.0 &      NaN &      NaN &   1.0 &    NaN &      1.0 \\
\bottomrule
\end{tabular}
}
\end{table}

\begin{table}
\centering
\caption{First consistent support identification.}
\label{tab:first_cst_id}
\scalebox{0.8}{
\begin{tabular}{lrrrrr}
\toprule
          instance & \proxsvrg{} & \saga{} & \rda{} & \pstorm{} & \spstorm{} \\
\midrule
          a9a-0.1-0.25 &     25.0 &     14.0 &   1.0 &  102.0 &     13.0 \\
           a9a-0.1-0.5 &    211.0 &    108.0 &   NaN &    NaN &    106.0 \\
          a9a-0.1-0.75 &     69.0 &     36.0 &   3.0 &  454.0 &     35.0 \\
           a9a-0.1-1.0 &     69.0 &     36.0 &   3.0 &  454.0 &     35.0 \\
         a9a-0.01-0.25 &      NaN &      NaN &   NaN &    NaN &    908.0 \\
          a9a-0.01-0.5 &    423.0 &    213.0 & 299.0 &    NaN &    215.0 \\
         a9a-0.01-0.75 &    349.0 &    176.0 &   NaN &    NaN &    176.0 \\
          a9a-0.01-1.0 &    349.0 &    176.0 &   NaN &    NaN &    176.0 \\
      covtype-0.1-0.25 &      3.0 &      3.0 &   1.0 &   11.0 &      1.0 \\
       covtype-0.1-0.5 &      1.0 &      3.0 &   1.0 &    5.0 &      1.0 \\
      covtype-0.1-0.75 &      1.0 &      3.0 &   1.0 &    5.0 &      1.0 \\
       covtype-0.1-1.0 &      1.0 &      3.0 &   1.0 &    5.0 &      1.0 \\
     covtype-0.01-0.25 &      1.0 &      1.0 &   1.0 &    1.0 &      1.0 \\
      covtype-0.01-0.5 &      1.0 &      1.0 &   1.0 &    1.0 &      1.0 \\
     covtype-0.01-0.75 &     29.0 &     24.0 &   8.0 &  879.0 &      1.0 \\
      covtype-0.01-1.0 &     29.0 &     24.0 &   8.0 &  879.0 &      1.0 \\
     phishing-0.1-0.25 &    209.0 &    105.0 &  10.0 &    NaN &    106.0 \\
      phishing-0.1-0.5 &    611.0 &    307.0 & 100.0 &    NaN &    306.0 \\
     phishing-0.1-0.75 &    581.0 &    291.0 & 520.0 &    NaN &    291.0 \\
      phishing-0.1-1.0 &    675.0 &    339.0 & 997.0 &    NaN &    334.0 \\
          w8a-0.1-0.25 &     41.0 &     22.0 &   4.0 &  252.0 &     21.0 \\
           w8a-0.1-0.5 &     95.0 &     49.0 & 110.0 &  881.0 &     47.0 \\
          w8a-0.1-0.75 &    115.0 &     59.0 &  65.0 &    NaN &     59.0 \\
           w8a-0.1-1.0 &    115.0 &     59.0 &  65.0 &    NaN &     59.0 \\
          w8a-0.01-0.5 &    837.0 &    420.0 &   NaN &    NaN &    421.0 \\
         w8a-0.01-0.75 &      NaN &    816.0 &   NaN &    NaN &    806.0 \\
          w8a-0.01-1.0 &      NaN &    816.0 &   NaN &    NaN &    806.0 \\
     real-sim-0.1-0.25 &      NaN &      NaN &   NaN &    NaN &    789.0 \\
      real-sim-0.1-0.5 &    675.0 &      NaN & 838.0 &    NaN &    337.0 \\
       news20-0.1-0.25 &    963.0 &      NaN &   NaN &    NaN &      NaN \\
        news20-0.1-0.5 &    995.0 &      NaN &   NaN &    NaN &    523.0 \\
 url-combined-0.1-0.25 &      9.0 &      NaN &   NaN &    NaN &      5.0 \\
 avazu-app.tr-0.1-0.25 &      3.0 &      NaN &   1.0 &    1.0 &      1.0 \\
  avazu-app.tr-0.1-0.5 &      3.0 &      NaN &   1.0 &    1.0 &      1.0 \\
 avazu-app.tr-0.1-0.75 &      3.0 &      NaN &   1.0 &    1.0 &      1.0 \\
  avazu-app.tr-0.1-1.0 &      NaN &      NaN &   1.0 &    1.0 &      1.0 \\
avazu-app.tr-0.01-0.25 &      3.0 &      NaN &   4.0 &    NaN &      1.0 \\
 avazu-app.tr-0.01-0.5 &      3.0 &      NaN &   4.0 &    NaN &      1.0 \\
avazu-app.tr-0.01-0.75 &      3.0 &      NaN &   2.0 &    NaN &      1.0 \\
 avazu-app.tr-0.01-1.0 &      NaN &      NaN &   1.0 &    NaN &      1.0 \\
\bottomrule
\end{tabular}
}
\end{table}

\begin{table}
\centering
\caption{Last iterate sparsity.}
\label{tab:last_nz}
\scalebox{0.8}{
\begin{tabular}{lrrrrrr}
\toprule
              instance & FaRSAGroup & \proxsvrg{} & \saga{} &   \rda{} & \pstorm{} & \spstorm{} \\
\midrule
          a9a-0.1-0.25 &         26 &        26.0 &    26.0 &     26.0 &      26.0 &       26.0 \\
           a9a-0.1-0.5 &         57 &        57.0 &    57.0 &     56.0 &      56.0 &       57.0 \\
          a9a-0.1-0.75 &         86 &        86.0 &    86.0 &     86.0 &      86.0 &       86.0 \\
           a9a-0.1-1.0 &        117 &       117.0 &   117.0 &    117.0 &     117.0 &      117.0 \\
         a9a-0.01-0.25 &         20 &        19.0 &    19.0 &     18.0 &      16.0 &       20.0 \\
          a9a-0.01-0.5 &         44 &        44.0 &    44.0 &     44.0 &      38.0 &       44.0 \\
         a9a-0.01-0.75 &         65 &        65.0 &    65.0 &     66.0 &      58.0 &       65.0 \\
          a9a-0.01-1.0 &         96 &        96.0 &    96.0 &     97.0 &      89.0 &       96.0 \\
      covtype-0.1-0.25 &         11 &        11.0 &    11.0 &     11.0 &      11.0 &       11.0 \\
       covtype-0.1-0.5 &         25 &        25.0 &    25.0 &     25.0 &      25.0 &       25.0 \\
      covtype-0.1-0.75 &         38 &        38.0 &    38.0 &     38.0 &      38.0 &       38.0 \\
       covtype-0.1-1.0 &         52 &        52.0 &    52.0 &     52.0 &      52.0 &       52.0 \\
     covtype-0.01-0.25 &         10 &        10.0 &    10.0 &     10.0 &      10.0 &       10.0 \\
      covtype-0.01-0.5 &         22 &        22.0 &    22.0 &     22.0 &      22.0 &       22.0 \\
     covtype-0.01-0.75 &         33 &        33.0 &    33.0 &     33.0 &      33.0 &       33.0 \\
      covtype-0.01-1.0 &         47 &        47.0 &    47.0 &     47.0 &      47.0 &       47.0 \\
     phishing-0.1-0.25 &         12 &        12.0 &    12.0 &     12.0 &       9.0 &       12.0 \\
      phishing-0.1-0.5 &         25 &        25.0 &    25.0 &     25.0 &      22.0 &       25.0 \\
     phishing-0.1-0.75 &         43 &        43.0 &    43.0 &     43.0 &      41.0 &       43.0 \\
      phishing-0.1-1.0 &         59 &        59.0 &    59.0 &     59.0 &      55.0 &       59.0 \\
          w8a-0.1-0.25 &         57 &        57.0 &    57.0 &     57.0 &      57.0 &       57.0 \\
           w8a-0.1-0.5 &        132 &       132.0 &   132.0 &    132.0 &     132.0 &      132.0 \\
          w8a-0.1-0.75 &        208 &       208.0 &   208.0 &    208.0 &     206.0 &      208.0 \\
           w8a-0.1-1.0 &        281 &       281.0 &   281.0 &    281.0 &     279.0 &      281.0 \\
          w8a-0.01-0.5 &         79 &        79.0 &    79.0 &     77.0 &      67.0 &       79.0 \\
         w8a-0.01-0.75 &        150 &       149.0 &   150.0 &    149.0 &     129.0 &      150.0 \\
          w8a-0.01-1.0 &        214 &       213.0 &   214.0 &    213.0 &     189.0 &      214.0 \\
     real-sim-0.1-0.25 &       5211 &      5210.0 &     NaN &   5210.0 &    5187.0 &     5211.0 \\
      real-sim-0.1-0.5 &      10439 &     10439.0 &     NaN &  10439.0 &   10394.0 &    10439.0 \\
       news20-0.1-0.25 &     338697 &    338697.0 &     NaN & 338661.0 &  338674.0 &   338698.0 \\
        news20-0.1-0.5 &     677496 &    677496.0 &     NaN & 677459.0 &  677464.0 &   677496.0 \\
 url-combined-0.1-0.25 &     807983 &    807983.0 &     NaN & 807982.0 &  807974.0 &   807983.0 \\
 avazu-app.tr-0.1-0.25 &     249995 &    249995.0 &     NaN & 249995.0 &  249995.0 &   249995.0 \\
  avazu-app.tr-0.1-0.5 &     499993 &    499993.0 &     NaN & 499993.0 &  499993.0 &   499993.0 \\
 avazu-app.tr-0.1-0.75 &     749990 &    749990.0 &     NaN & 749990.0 &  749990.0 &   749990.0 \\
  avazu-app.tr-0.1-1.0 &     999988 &    999987.0 &     NaN & 999988.0 &  999988.0 &   999988.0 \\
avazu-app.tr-0.01-0.25 &     249980 &    249980.0 &     NaN & 249980.0 &  249973.0 &   249980.0 \\
 avazu-app.tr-0.01-0.5 &     499978 &    499978.0 &     NaN & 499979.0 &  499970.0 &   499978.0 \\
avazu-app.tr-0.01-0.75 &     749976 &    749976.0 &     NaN & 749976.0 &  749972.0 &   749976.0 \\
 avazu-app.tr-0.01-1.0 &     999973 &    999814.0 &     NaN & 999973.0 &  999965.0 &   999973.0 \\
\bottomrule
\end{tabular}
}
\end{table}

\end{document}